\documentclass{amsart}
\usepackage{amsmath,amssymb,amsfonts}
\usepackage[mathscr]{eucal}

\usepackage{enumerate}
\newenvironment{enumroman}{\begin{enumerate}[\upshape (i)]}
                                                {\end{enumerate}}

\theoremstyle{plain}
\newtheorem{theorem}{Theorem}[section]

\newtheorem{prop}[theorem]{Proposition}

\theoremstyle{definition}
\newtheorem{definition}[theorem]{Definition}
\newtheorem{example}[theorem]{Example}

\newtheorem{remark}[theorem]{Remark}

\newcommand{\Deltaop}{{\bf \Delta}^{op}}
\newcommand{\id}{\text{id}}
\newcommand{\Hom}{\text{Hom}}
\newcommand{\Map}{\text{Map}}
\newcommand{\Aut}{\text{Aut}}
\newcommand{\map}{\text{map}}
\newcommand{\ob}{\text{ob}}
\newcommand{\iso}{\text{iso}}
\newcommand{\Sets}{\mathcal Sets}
\newcommand{\SSets}{\mathcal{SS}ets}
\newcommand{\Secat}{\mathcal Se \mathcal Cat}
\newcommand{\css}{\mathcal{CSS}}
\newcommand{\Qcat}{\mathcal{QC}at}
\newcommand{\nerve}{\text{nerve}}
\newcommand{\Sing}{\text{Sing}}
\newcommand{\Ho}{\text{Ho}}
\newcommand{\cosk}{\text{cosk}}
\newcommand{\Ext}{\text{Ext}}
\newcommand{\Hoch}{\text{Hoch}}
\newcommand{\End}{\text{End}}
\newcommand{\Fun}{\text{Fun}}
\newcommand{\Alg}{\text{Alg}}

\input xy
\xyoption{all}

\begin{document}
\title{Workshop on the Homotopy Theory of Homotopy Theories}

\author[J.E. Bergner]{Julia E. Bergner}

\address{Department of Mathematics, University of California, Riverside, CA 92521}

\email{bergnerj@member.ams.org}

\date{\today}

\thanks{The author was partially supported by NSF grant DMS-0805951.}

\begin{abstract}
These notes are from a series of lectures given at the Workshop on the Homotopy Theory of Homotopy Theories which took place in Caesarea, Israel, in May 2010.  The workshop was organized by David Blanc, Emmanuel Farjoun, and David Kazhdan, and talks not indicated otherwise were given by the author.
\end{abstract}

\maketitle

\section{Overview: Introduction to the homotopy theory of homotopy theories}

To understand homotopy theories, and then the homotopy theory of them, we first need an understanding of what a ``homotopy theory" is.  The starting point is the classical homotopy theory of topological spaces.  In this setting, we consider topological spaces up to homotopy equivalence, or up to weak homotopy equivalence.  Techniques were developed for defining a nice homotopy category of spaces, in which we define morphisms between spaces to be homotopy classes of maps between CW complex replacements of the original spaces being considered.

However, the general framework here is not unique to topology; an analogous situation can be found in homological algebra.  We can take projective replacements of chain complexes, then chain homotopy classes of maps, to define the derived category, the algebraic analogue of the homotopy category of spaces.

The question of when we can make this kind of construction (replacing by some particularly nice kinds of objects and then taking homotopy classes of maps) led to the definition of a model category by Quillen in the 1960s \cite{quillen}.  The essential information consists of some category of mathematical objects, together with some choice of which maps are to be considered weak equivalences.  The additional information, and the axioms this data must satisfy, guarantee the existence of a well-behaved homotopy category as we have in the above examples, with no set theory problems arising.

A more general notion of homotopy theory was developed by Dwyer and Kan in the 1980s.  Their simplicial localization \cite{dksimploc} and hammock localization \cite{dkfncxes} techniques provided a method in which a category with weak equivalences can be assigned to a simplicial category.  More remarkably, they showed that up to a natural notion of equivalence (now called Dwyer-Kan equivalence), every simplicial category arises in this way \cite{dkdiag}.  Thus, if a ``homotopy theory" is just a category with weak equivalences, then we can just as easily think of simplicial categories as homotopy theories.  In other words, simplicial categories provide a model for homotopy theories.

However, with Dwyer-Kan equivalences, the category of (small) simplicial categories itself forms a category with weak equivalences, and is therefore a homotopy theory.  Hence, we have a ``homotopy theory of homotopy theories," the topic of this workshop.  In fact, this category was shown to have a model structure in 2004, making it a homotopy theory in the more rigorous sense \cite{simpcat}.  (In more recent work, Barwick and Kan have given ``categories with weak equivalences" a model structure which is equivalent to this one \cite{bk}.)

In practice, unfortunately, this model structure is not as nice as a homotopy theorist might wish.  It is not compatible with the monoidal structure on the category of simplicial categories, does not seem to have the structure of a simplicial model category in any natural way, and has weak equivalences which are difficult to identify for any given example.  Therefore, a good homotopy theorist might seek an equivalent model structure with better properties.

An alternate model, that of complete Segal spaces, was proposed by Rezk in 2000, in fact before the model structure on simplicial categories was established \cite{rezk}.  Complete Segal spaces are just simplicial diagrams of simplicial sets, satisfying some conditions which allow them to be thought of as something like simplicial categories but with weak composition.  Their corresponding model category is cartesian closed, and is in fact given by a localization of the Reedy model structure on simplicial spaces \cite{reedy}.  Hence, the weak equivalences between fibrant objects are just levelwise weak equivalences of simplicial sets.

Meanwhile, in the world of category theory, simplicial categories were seen as models for $(\infty,1)$-categories.  These particular examples of $\infty$-categories, with $k$-morphisms defined for all $k \geq 1$, satisfy the property that for $k>1$, the $k$-morphisms are all weakly invertible.  To see why simplicial categories were a natural model, it is perhaps easier to consider instead topological categories, where we have a topological space of morphisms between any two objects.  The 1-morphisms are just points in these mapping spaces.  The 2-morphisms are paths between these points; at least up to homotopy, they are invertible.  Then 3-morphisms are homotopies between paths, 4-morphisms are homotopies between homotopies, and one could continue indefinitely.

In the 1990s, Segal categories were developed as a weakened version of simplicial categories.  They are simplicial spaces with discrete 0-space, and look like homotopy versions of the nerves of simplicial categories.  They were first defined by Dwyer-Kan-Smith \cite{dks}, but developed from this categorical perspective by Hirschowitz and Simpson \cite{hs}.  The model structure for Segal categories, begun in their work, was given explicitly by Pelissier \cite{pel}.

Yet another model for $(\infty,1)$-categories was given in the form of quasi-categories or weak Kan complexes, first defined by Boardman and Vogt \cite{bv}.  They were developed extensively by Joyal, who defined many standard categorical notions, for example limits and colimits, within this more general setting \cite{joyal}, \cite{joyalbook}.  The notion was adopted by Lurie, who established many of Joyal's results independently \cite{lurie}.

Comparisons between all these various models were conjectured by several people, including To\"en \cite{toen} and Rezk \cite{rezk}.  In fact, To\"en proved that any model category satisfying a particular list of axioms must be Quillen equivalent to the complete Segal space model structure, hence axiomatizing what it meant to be a homotopy theory of homotopy theories, or homotopy theory of $(\infty, 1)$-categories \cite{toencss}.

Eventually, explicit comparisons were made, as shown in the following diagram:
\[ \xymatrix{\mathcal{SC} \ar@<-.5ex>[r] \ar@<-.5ex>[drr] & \Secat_f \ar@<.5ex>[r] \ar@<-.5ex>[l]  & \Secat_c \ar@<.5ex>[r] \ar@<.5ex>[l] \ar@{<->}[d] & \css \ar@<.5ex>[l] \ar@{<->}[ld] \\
&& \Qcat \ar@{<->}[u] \ar@{<->}[ur] \ar@<-.5ex>[ull]} \]
The single arrows indicate that Quillen equivalences were given in both directions, and these were established by Joyal and Tierney \cite{jt}.  The Quillen equivalence between simplicial categories and quasi-categories was proved in different ways by Joyal \cite{joyalbook}, Lurie \cite{lurie}, and Dugger-Spivak \cite{dugspi}, \cite{dugspi2}.  The zig-zag across the top row was established by the author \cite{thesis}.  It should be noted that the additional model structure $\mathcal Se \mathcal Cat_f$ for Segal categories was established for the purposes of this proof; the original one is denoted $\mathcal Se \mathcal Cat_c$.

At this point, one might ask what these ideas are good for.  Are they just a nice way of expressing abstract ideas, or do they help us to understand specific mathematical situations?  The answer is that they are being used in a multitude of areas.  We list several here, but do not claim that this list is exhaustive.

\begin{enumerate}
\item Algebraic geometry

\noindent
Many of the models are being used for various purposes in algebraic geometry.  Authors such as Simpson, To\"en, and Vezzosi are using Segal categories \cite{hs}, \cite{tv}, Lurie is using quasi-categories \cite{luriestable}, \cite{dag3}, and Barwick is using complete Segal spaces and quasi-categories \cite{bar}.

\item $K$-theory

\noindent
Simplicial categories are used frequently in this area; we mention Blumberg-Mandell \cite{bm} and To\"en-Vezzosi \cite{tvk}; the latter mention ways to generalize their work to Segal categories as well.

\item Representation theory

\noindent
Quasi-categories are being used, for example by Ben-Zvi, Francis, and Nadler \cite{bzn1}, \cite{bzn2}, \cite{bzn3}, \cite{bzn4}, \cite{bzfn}.  The author has done some work using complete Segal spaces \cite{dhacss}.

\item Deformation theory

\noindent
Lurie is using quasi-categories \cite{dag6}, and Pridham has developed an application of complete Segal spaces \cite{prid}.

\item Homotopy theory

\noindent
Lurie has used quasi-categories to prove results in stable homotopy theory \cite{luriesh}, and other models for Goodwillie calculus \cite{lurie2}.

\item Topological field theories

\noindent
Lurie has used complete Segal spaces to prove the Cobordism Hypothesis \cite{lurietft}.

\end{enumerate}

Two of these last examples, Goodwillie calculus and the Cobordism Hypothesis, are actually applications of $(\infty, n)$-categories, in which $k$-morphisms are invertible for $k>n$, not just the case where $n=1$.  Developing models for these more general structures has begun but is still very much a work in progress, especially in terms of comparisons.

The earliest model for $(\infty, n)$-categories was that of Segal $n$-categories; in fact the work of Hirschowitz-Simpson \cite{hs} and Pelissier \cite{pel} is in this generality.  In the last few years, Barwick developed $n$-fold complete Segal spaces, and this was the model used by Lurie \cite{lurietft}.  Then Rezk defined $\Theta_n$-spaces, an alternative generalization of complete Segal spaces \cite{rezktheta}.  He conjectured that one could then enrich categories over $\Theta_{n-1}$-spaces, find Segal category analogues (which would necessarily be somewhat different from the Segal $n$-category versions), and then do a comparison along the lines of what was done in the $n=1$ case.  Developing and comparing all these models is work in progress of the author and Rezk, with similar work being done by Tomesch.

With this overview and motivation given, we turn to understanding many of these ideas in more detail in the remaining talks of the workshop.

\section{Models for homotopy theories}

\centerline{(Talk given by Boris Chorny)}

\vskip .1 in

Take your favorite mathematical objects.  If morphisms satisfy associativity and identity morphisms exist, then we have a category $\mathcal C$.  We can take some collection $S$ of morphisms of $\mathcal C$, called weak equivalences, that we want to think of as isomorphisms but may not actually be so.

We can look at Gabriel-Zisman localization \cite{gz}.  Define a map
\[ \gamma \colon \mathcal C \rightarrow \mathcal C[S^{-1}] \] such that the maps in $S$ go to isomorphisms.  The category $\mathcal C[S^{-1}]$ has more maps than the original $\mathcal C$.  For example, zig-zags with backward maps in $S$ become morphisms in $\mathcal C[S^{-1}]$.  The category $\mathcal C[S^{-1}]$ is also required to satisfy the following universal property: if $\delta \colon \mathcal C \rightarrow \mathcal D$ is a functor such that $\delta(s)$ is an isomorphism in $\mathcal D$ for every map $s$ of $S$, then there is a unique map $\mathcal C[S^{-1}] \rightarrow \mathcal D$ such that the diagram
\[ \xymatrix{\mathcal C \ar[rr]^\gamma \ar[dr]^\delta && \mathcal C[S^{-1}] \ar[dl] \\
& \mathcal D & } \] commutes.

\begin{example}
If $\mathcal C= \mathcal Gps$ is the category of groups, we can define
\[ S=\{ G_1 \rightarrow G_2 \mid G_1/[G_1, G_1] \rightarrow G_2/[G_2, G_2] \} \] where the bracket denotes the commutator subgroup.  Then $\mathcal Gps[S^{-1}] \simeq \mathcal Ab$, the category of abelian groups.
\end{example}

\begin{example}
If $\mathcal C = \mathcal Top$ is the category of topological spaces, we can consider $S$ to be the class of homotopy equivalences, or $S_1$ the class of weak homotopy equivalences.  There are thus two homotopy categories to consider, $\mathcal Top[S^{-1}]$, which is too complicated, and $\mathcal Top[S_1^{-1}]$, which is usually considered as the homotopy category.
\end{example}

\begin{example}
If $\mathcal C=Ch_{\geq 0}(\mathbb Z)$ is the category of non-negatively graded chain complexes of abelian groups, let $S$ be the class of quasi-isomorphisms.  Then $Ch_{\geq 0}(\mathbb Z)[S^{-1}]= \mathcal D^{\geq 0}(\mathbb Z)$, the bounded below derived category.
\end{example}

In general, the hom-sets in $\mathcal C[S^{-1}]$ may be too big, in that they may form proper classes rather than sets.  There are three possible solutions to this difficulty.
\begin{itemize}
\item Use the universe axiom, so that everything ``large" becomes small in the next universe.  However, this process causes some changes in the original model.

\item Sometimes Ore's condition holds, so any diagram
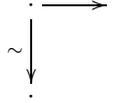
\[ \xymatrix{ \cdotp \ar[r] \ar[d]_\sim & \cdotp \\ \cdotp & } \] can be completed to a diagram
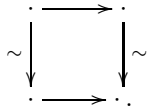
\[ \xymatrix{ \cdotp \ar[r] \ar[d]_\sim & \cdotp \ar[d]^ \sim \\
\cdotp \ar[r] & \cdotp .} \]  For example, this condition helps in the construction of the derived category.

\item Consider a model category structure on $\mathcal C$.
\end{itemize}

To get a model category structure, take subcategories of $\mathcal C$: $\mathcal W$, the weak equivalences, $\mathcal Fib$, the fibrations, and $\mathcal Cof$, the cofibrations, each containing all isomorphisms.  This structure, defined by Quillen in 1967 \cite{quillen}, must satisfy five axioms.

These axioms actually guarantee that every two classes together determine the third.  If we have a diagram
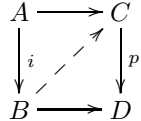
\[ \xymatrix{A \ar[r] \ar[d]^i & C \ar[d]^p \\
B \ar[r] \ar@{-->}[ur] & D} \] then if the dotted arrow lift exists, we say that $i$ has the right lifting property with respect to $p$.  If we choose weak equivalences and one of the other classes, then the third is determined by the lifting properties given by axiom 5: $\mathcal W \cap \mathcal Cof$ has the left lifting property with respect to $\mathcal Fib$, and $\mathcal Cof$ has the left lifting property with respect to $\mathcal Fib \cap \mathcal W$.  If the fibrations and cofibrations are specified, then the weak equivalences are determined by this axiom together with the fact that any weak equivalence must be able to be factored as an acyclic cofibration (map which is a cofibration and weak equivalence) followed by an acyclic fibration.

\begin{example}
Our two choices of weak equivalences in $\mathcal Top$ give two possible model structures.  If the weak equivalences are homotopy equivalences, we can define the fibrations to be the Hurewicz fibrations and the cofibrations to be the Borsuk pairs.  The result is Str{\o}m's model structure, which is complicated \cite{strom}.

If we take the weak equivalences to be the weak homotopy equivalences, the cofibrations to be the class of CW inclusions and their retracts, and the fibrations to be Serre fibrations, then we get the standard model structure \cite{ds}.
\end{example}

\begin{example}
For $Ch_{\geq 0}(\mathbb Z)$ with weak equivalences the quasi-isomorphisms, then we can choose the cofibrations to be the injections and the fibrations to be the surjections with projective kernels.  Alternatively, one could take the same weak equivalences but choose the fibrations to be all surjections and the cofibrations to be the injections with injective cokernel \cite{ds}.
\end{example}

Models for a given homotopy theory are not necessarily unique.  For example, there is a model category structure on the category $\SSets$ of simplicial sets whose homotopy category is equivalent to $\mathcal Top[\mathcal W^{-1}]$.  There is in fact a Quillen equivalence between $\mathcal Top$ and $\SSets$: an adjunction such that the left adjoint preserves (acyclic) cofibrations and the right
adjoint preserves (acyclic) fibrations.

These techniques can be applied in other settings.

\begin{example}
In Voevodsky's work (e.g., in \cite{mv}), he uses the Yoneda embedding $Sm/\Delta \rightarrow \SSets(Sm/\Delta)^{op}$.
\end{example}

\begin{example}
The Adams spectral sequence computes stable homotopy classes of maps:
\[ Ext_\mathcal A(H^*Y, H^*X) \Rightarrow [X,Y]_{st}. \]  The Bousfield-Kan spectral sequence is an unstable version.  Farjoun and Zabrodsky generalized it to compute equivariant homotopy classes of maps \cite{fz}.  In 2004, Bousfield generalized for abstract model categories \cite{bousfield}.
\end{example}

\section{Simplicial objects and categories}

\centerline{(Talk given by Emmanuel Farjoun)}

\vskip .1 in

The goals of this talk are to explain:
\begin{itemize}
\item why simplicial objects are so often met, and

\item why work with simplicial categories and how to get them.
\end{itemize}

Recall the category ${\bf \Delta}$, which looks like
\[ \xymatrix@1{[0] \ar@<.8ex>[r] \ar@<-.8ex>[r] & [1] \ar[l] \ar@<1ex>[r] \ar[r] \ar@<-1ex>[r] & [2] \cdots. \ar@<.5ex>[l] \ar@<-.5ex>[l]} \]
The object $[n]$ is the ordered set $\{0, \ldots, n\}$ and the maps are order-preserving and satisfy simplicial relations.  We can also take the opposite category $\Deltaop$.

\begin{definition}
A \emph{simplicial object} in a category $\mathcal C$ is a functor $\Deltaop \rightarrow \mathcal C$.
\end{definition}

Historically, Eilenberg and Mac Lane were interested in continuous maps $\Delta^n \rightarrow X$ where $\Delta^n$ is the $n$-simplex in $\mathbb R^n$ and $X$ is a topological space.  But why ${\bf \Delta}$?

The slogan is that ${\bf \Delta}$ (or $\Deltaop$) is the category of endomorphisms of augmented functors.  The objects of this category are functors $a \colon X \rightarrow FX$.  We want to ``recover" $X$ out of $FX$, $F^2X$, $F^3X$, etc.  (This process could be regarded as similar to taking a field extension $k \rightarrow \overline k$ and trying to recover the original $k$.)

What are the natural transformations between powers of $X$?  For example, the two maps $FX \Rightarrow F^2X$ are given by $Fa_X$ and by $a_{FX}$.

\begin{example}
Start with the empty space $\phi$ and add a point to make it simpler: $\phi^+$.  (Now it is contractible.)  We can continue this process of adding a point to get
\[ \xymatrix@1{\phi \ar[r] & \bullet \ar@<.8ex>[r] \ar@<-.8ex>[r] & (\bullet \rightarrow \bullet) \ar[l] \ar@<1ex>[r] \ar[r] \ar@<-1ex>[r] & \Delta \cdots.\ar@<.5ex>[l] \ar@<-.5ex>[l]} \]
We also get a map $F^2X \rightarrow FX$, since FX is contractible.  Iterating, we just get the category ${\bf \Delta}$.
\end{example}

However, this process does not work for all functors.

\begin{example}
In algebra, consider a module $M$.  We get a diagram
\[ M \leftrightarrows FM \Leftarrow F^2M \cdots \] where the map $i \colon M \rightarrow FM$ sends $m \mapsto m \cdotp 1$, and hence there is also a map $Fi \colon FM \rightarrow F^2M$.  Sometimes you can recover the module $M$ from this data, but sometimes not.
\end{example}

Recall from the previous talk that a category $\mathcal C$ with weak equivalences gives rise to $\mathcal C[S^{-1}]$, its localization with respect to $S$.  The difficulty with this construction is that we don't keep the limits and colimits that we had in $\mathcal C$.  For example, suppose we had a pullback diagram in $\mathcal C$:
\[ \xymatrix{P \ar[r] \ar[d] & X \ar[d] \\
Y \ar[r] & B.} \]
Then in the homotopy category, using subscripts to denote the set of maps being inverted, the diagram
\[ \xymatrix{[W,P]_{S^{-1}} \ar[r] \ar[d] & [W,X]_{S^{-1}} \ar[d] \\
[W,Y]_{S^{-1}} \ar[r] & [W,B]_{S^{-1}} } \] is not in general a pullback for any object $W$ in $\mathcal C$.  We want to correct this construction so that it is, taking mapping spaces instead of mapping sets:
\[ \xymatrix{\Map_{S^{-1}}(W,P) \ar[r] \ar[d] & \Map_{S^{-1}}(W,X) \ar[d] \\
\Map_{S^{-1}}(W,Y) \ar[r] & \Map_{S^{-1}}(W,B)} \] which is a homotopy pullback of spaces.  Applying $\pi_0$ gives the previous diagram, but now we can see that it sits in a long exact sequence, so we can see why the above process did not work: there must be correction terms given by higher homotopy groups.

So, how do we define these mapping spaces?  The process is called \emph{simplicial localization} \cite{dksimploc}.

Given a category $\mathcal C$ and a subcategory $S$ of ``equivalences" with $\ob(S)=\ob(\mathcal C)$, we want a simplicial category $L(\mathcal C, S)$. Recall that a simplicial category is a category enriched in simplicial sets, so it has objects and, for any pair of objects $A$ and $B$, $\Map(A,B)$ is a simplicial set.  We can think of it as the ``derived functor of a non-additive functor $F$:
\[ \cdots F^3 \mathcal C \Rrightarrow F^2 \mathcal C \Rightarrow F \mathcal C \rightarrow \mathcal C \supset S. \]
If we invert $S$ at each level, we get a diagram
\[ \cdots F^3 \mathcal C[S^{-1}] \Rrightarrow F^2 \mathcal C[S^{-1}] \Rightarrow F \mathcal C[S^{-1}]. \]

\begin{theorem} \cite{dksimploc}
$\pi_0 L(\mathcal C,S) = \mathcal C[S^{-1}]$.
\end{theorem}

Besides universality, the fantastic property of this simplicial localization, or justification for it, as follows.  Suppose we start with a model category $\mathcal M$.  We can use the model structure to get $\Map(X,Y)$ using the cone or cylinder construction, for example $\Map(X,Y)_1 = \Hom_\mathcal M(I \otimes X, Y)$.  The resulting simplicial category is equivalent to $L(\mathcal M, S)$.

Universality works the same here as it does for the homotopy category: if we have a functor $T \colon (\mathcal C, S) \rightarrow \mathcal D$, where $\mathcal D$ is a simplicial category, such that $T(s)$ is an equivalence for every $s$ in $S$, then there exists a unique simplicial functor $L(\mathcal C,S) \rightarrow \mathcal D$ making the following diagram commute:
\[ \xymatrix{(\mathcal C, S) \ar[rr] \ar[dr] && L(\mathcal C,S) \ar@{-->}[dl] \\
& \mathcal D. &} \]

Under some conditions, we can equivalently write $L(\mathcal C,S)$ as the simplicial category with $L(\mathcal C,S)_0$ consisting of diagrams
\[ \xymatrix@1{A & \cdotp \ar[l]_s \ar[r]^f & \cdotp & B \ar[l]_{s'}} \] and $L(\mathcal C, S)_1$ consisting of diagrams
\[ \xymatrix{& \cdotp \ar[dl]_\sim \ar[r]^f \ar[dd]^\sim & \cdotp \ar[dd]^\sim & \\
A &&& B. \ar[ul]_\sim \ar[dl]_\sim \\
& \cdotp \ar[r]^{f'} \ar[ul]_\sim & \cdotp &} \]
This is sometimes called the \emph{hammock localization} \cite{dkfncxes}.

\section{Nerve constructions}

\begin{definition}
The \emph{nerve} of a small category $\mathcal C$ is the simplicial set $\nerve(\mathcal C)$ defined by
\[ \nerve(\mathcal C)_n = \Hom_{\mathcal Cat}([n], \mathcal C). \]
\end{definition}

In other words, the $n$-simplices of the nerve are $n$-tuples of composable morphisms in $\mathcal C$.  Recall that if $f \colon \mathcal C \rightarrow \mathcal D$ is an equivalence of categories, then $f$ induces a weak equivalence $\nerve(\mathcal C) \rightarrow \nerve(\mathcal D)$.

\begin{example}
Let $\mathcal C$ be the category with a single object and only the identity morphism, and let $\mathcal D$ be the category with two objects and a single isomorphism between them.  Then we can include $\mathcal C$ into $\mathcal D$ (in two different ways), and this functor is an equivalence of categories.  The nerves of the categories $\mathcal C$ and $\mathcal D$ are both contractible, so we get a homotopy equivalence between them.
\end{example}

However, the converse statement is not true!

\begin{example}
Let $\mathcal E$ be the category with two objects, $x$ and $y$, and a single morphism from $x$ to $y$ and no other non-identity morphisms.  Then $\nerve(\mathcal E)$ is contractible, but $\mathcal E$ is not equivalent to either $\mathcal C$ or $\mathcal D$ from the previous example.
\end{example}

The problem here is that weak equivalences of simplicial sets are given by weak homotopy equivalences of spaces after geometric realization, so we don't remember which direction a 1-simplex pointed.  In particular, we don't remember whether a 1-simplex in the nerve came from an isomorphism in the original category.  However, if we restrict to groupoids, where all morphisms are isomorphisms, then the converse statement is true.

In this talk, we'll consider two possible approaches to take after making these observations.
\begin{enumerate}
\item We'll look at how to determine whether a simplicial set is the nerve of a groupoid or, more generally, the nerve of a category.

\item We'll define a more refined version of the nerve construction so that it does distinguish isomorphisms from other morphisms.
\end{enumerate}

\subsection{Kan complexes and generalizations}

Recall some particular useful examples of simplicial sets: the $n$-simplex $\Delta[n]= \Hom(-, [n])$; its boundary $\partial \Delta[n]$, obtained by leaving out the identity map $[n] \rightarrow [n]$; and the horns $V[n,k]$, which are obtained by removing the $k$th face from $\partial \Delta[n]$.

For example, when $n=2$, we can think of $\partial \Delta[2]$ as the (not filled in) triangle
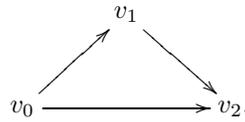
\[ \xymatrix{& v_1  \ar[dr] & \\
v_0 \ar[rr] \ar[ur] && v_2.} \]  Then the horn $V[2,0]$ looks like
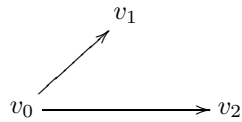
\[ \xymatrix{& v_1  & \\
v_0 \ar[rr] \ar[ur] && v_2} \] whereas $V[2,1]$ looks like
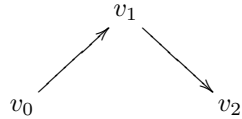
\[ \xymatrix{& v_1  \ar[dr] & \\
v_0 \ar[ur] && v_2} \] and $V[2,2]$ looks like
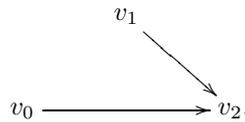
\[ \xymatrix{& v_1  \ar[dr] & \\
v_0 \ar[rr] && v_2.} \]

\begin{definition}
A simplicial set $X$ is a \emph{Kan complex} if any map $V[n,k] \rightarrow X$ can be extended to a map $\Delta[n] \rightarrow X$.  In other words, a lift exists in any diagram
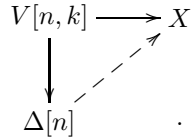
\[ \xymatrix{V[n,k] \ar[d] \ar[r] & X \\
\Delta[n] \ar@{-->}[ur] &. } \]
\end{definition}

\begin{prop}
The nerve of a groupoid is a Kan complex.
\end{prop}

The idea of the proof of this proposition is as follows.  Let $G$ be a groupoid.  When $n=2$, having a lift in the diagram
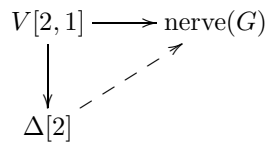
\[ \xymatrix{V[2,1] \ar[d] \ar[r] & \nerve(G) \\
\Delta[2] \ar@{-->}[ur] &} \] means that any pair of composable morphisms has a composite.  In fact, such a lift will exist for the nerve of any category, not just a groupoid.

However, the 1-simplices in the image of a map $V[2,0] \rightarrow \nerve(G)$ are not composable; they have a common source instead.  Therefore, finding a lift to $\Delta[2] \rightarrow X$ could be found if we knew that one of these 1-simplices came from an invertible morphism in $G$; since $G$ is a groupoid, this is always the case.  Finding a lift for the case $V[2,2]$ is similar; here the 1-simplices have a common target.

So, we have concluded that a lift exists in a diagram
\[ \xymatrix{V[n,k] \ar[r] \ar[d] & \nerve(\mathcal C) \\
\Delta[n] \ar@{-->}[ur] & } \] for $0\leq k \leq n$ if $\mathcal C$ is a groupoid but only for $0<k<n$ if $\mathcal C$ is any category.

However, since composition in a category is unique, lifts in the above diagram will also be unique.

\begin{prop}
A Kan complex $K$ is the nerve of a groupoid if and only if the lift in each diagram
\[ \xymatrix{V[n,k] \ar[r] \ar[d] & K \\
\Delta[n] \ar@{-->}[ur] & } \] is unique for every $0 \leq k \leq n$.
\end{prop}

Generalizing the above observation, we get the following.

\begin{definition}
A simplicial set $K$ is an \emph{inner Kan complex} or a \emph{quasi-category} if a lift exists in any diagram
\[ \xymatrix{V[n,k] \ar[d] \ar[r] & X \\
\Delta[n] \ar@{-->}[ur] & } \] for $0<k<n$.
\end{definition}

\begin{prop} \label{quasi}
A quasi-category is the nerve of a category if and only if these lifts are all unique.
\end{prop}

\subsection{Classifying diagrams}

Recall that a \emph{simplicial space} or \emph{bisimplicial set} is a functor $\Deltaop \rightarrow \SSets$.

\begin{definition} \cite{rezk}
The \emph{classifying diagram} $N\mathcal C$ of a category $\mathcal C$ is the simplicial space defined by
\[ (N\mathcal C)_n= \nerve(\iso(\mathcal C^{[n]})) \] where $\mathcal C^{[n]}$ denotes the category of functors $[n] \rightarrow \mathcal C$ whose objects are length $n$ chains of composable morphisms in $\mathcal C$, and where $\iso$ denotes the maximal subgroupoid functor.
\end{definition}

What does this definition mean?  When $n=0$, $(N\mathcal C)_0 = \nerve(\iso (\mathcal C))$ is the nerve of the maximal subgroupoid of $\mathcal C$.  In particular, this simplicial set only picks up information about isomorphisms in $\mathcal C$.  When $n=1$, we have $(N \mathcal C)_1= \nerve(\iso (\mathcal C^{[1]}))$.  The objects of $\iso(\mathcal C^{[1]})$ are morphisms in $\mathcal C$, and the morphisms of $\iso(\mathcal C^{[1]})$ are pairs of morphisms making the appropriate square diagram commute.  More generally, $(N \mathcal C)_{n,m}$ is the set of diagrams of the form
\[ \xymatrix{\cdotp \ar[r] \ar[d]^\cong & \cdotp \ar[r] \ar[d]^\cong & \cdots \ar[r] & \cdotp \ar[d]^\cong \\
\cdotp \ar[r] \ar[d]^\cong & \cdotp \ar[r] \ar[d]^\cong & \cdots \ar[r] & \cdotp \ar[d]^\cong \\
\vdots \ar[d]^\cong & \vdots \ar[d]^\cong & \ddots & \vdots \ar[d]^\cong \\
\cdotp \ar[r] & \cdotp \ar[r] & \cdots \ar[r] & \cdotp } \] where there are $n$ horizontal arrows in each row and $m$ vertical arrows in each column.

\begin{example}
If $\mathcal C$ is a groupoid, then up to homotopy, $N \mathcal C$ is a constant simplicial space and has the homotopy type of the nerve of $\mathcal C$.
\end{example}

\begin{example}
Let $\mathcal C$ be the category $\mathcal E$ from above.  Then $(N \mathcal C)_0$ has the homotopy type of two points and $(N \mathcal C)_1$ has the homotopy type of three points.  In particular, the two sets are not the same, and therefore the classifying diagram is not levelwise equivalent to the category with one object and one morphism, since it is a groupoid.
\end{example}

We have several nice facts about classifying diagrams of categories.
\begin{enumerate}
\item The homotopy types of the simplicial sets $(N \mathcal C)_n$ are determined by $(N \mathcal C)_0$ and $(N \mathcal C)_1$, in that
\[ (N \mathcal C)_n \simeq \underbrace{(N \mathcal C)_1 \times_{(N \mathcal C)_0} \cdots \times_{(N \mathcal C)_0} (N \mathcal C)_1}_n. \]
This idea will lead to the definition of Segal space.

\item The subspace of $(N \mathcal C)_1$ arising from isomorphisms in $\mathcal C$ is weakly equivalent to $(N \mathcal C)_0$.  This idea will lead to the definition of complete Segal space.
\end{enumerate}

At this point we can make a connection to the previous subsection.  Notice that the simplicial set $(N \mathcal C)_{\ast, 0}$ is just the nerve of $\mathcal C$.  In particular, it is a quasi-category.  This property will continue to hold for more general complete Segal spaces.

\section{Segal's approach to loop spaces}

\centerline{(Talk given by Matan Prezma)}

\vskip .1 in

We begin by recalling the definition of a loop space.

\begin{definition}
Given a space $Y$ with base point $y_0$, its \emph{loop space} is
\[ \Omega Y = \{\gamma \colon [0,1] \rightarrow Y \mid \gamma(0)=y_0 =\gamma(1) \}. \]
\end{definition}

Whenever we have loops $\alpha$ and $\beta$ in $\Omega Y$, we can compose them (via concatenation) to get a loop $\alpha \ast \beta$.  While this operation is not strictly associative, we do have that the loop $\alpha \ast (\beta \ast \gamma)$ is homotopic to the loop $(\alpha \ast \beta) \ast \gamma$.  In particular, if 1 denotes the constant loop, then for every $\alpha \in \Omega Y$, there exists a loop $\alpha^{-1}$ such that $\alpha \ast \alpha^{-1}$ is homotopic to 1.

\subsection{Work of Kan}

Consider $\Omega Y$ for a pointed connected space $Y$.  Applying the singular functor gives a simplicial set $\Sing(Y)$ which is homotopy equivalent to $\text{red}(\Sing(Y))$, a reduced simplicial set, or one with only one vertex.

For any reduced simplicial set $S$, we can apply the Kan loop group functor $G$ to get a simplicial group $GX$.  In particular, $(GS)_0=FS_1$, the free group on the set $S_1$.  (Note that the maps in this simplicial group are complicated.)

Putting these two ideas together, we can obtain that
\[ |G \text{ red}(\Sing Y)| \simeq \Omega Y. \]  In other words, the topological group on the left-hand side is a rigidification of $\Omega Y$.

\subsection{Work of Stasheff}

This work originates with the following question: When is a space $X$ homotopy equivalent to $\Omega Y$ for some space $Y$?  To answer this question, we get the idea of $A_\infty$-spaces.

The base case is where the space $X$ is equipped with a pointed map $\mu \colon W \times X \rightarrow X$, such that the base point of $X$ behaves like a unit element.  In this case we say that $X$ is an $A_2$-space.

We can then consider when this map $\mu$ comes equipped with a homotopy between $\mu \circ (1 \times \mu)$ and $\mu \circ (\mu \times 1)$, so that the diagram
\[ \xymatrix{X \times X \times X \ar[r]^-{\mu \times 1} \ar[d]_{1 \times \mu} & X \times X \ar[d]^{\mu} \\
X \times X \ar[r]^\mu & X} \] commutes up to homotopy.  In this case, $X$ is an $A_3$-space.

Now, given an $A_3$-space $X$, we could ask what the next level of associativity is, in order to define an $A_4$-space.  We would need to have a relationship between the different ways of parenthesizing four elements, say $\alpha, \beta, \gamma$, and $\delta$, so we need to consider the diagram
\[ \xymatrix{ && (\alpha \beta){\gamma \delta} \ar@{-}[drr] \ar@{-}[dll] && \\
((\alpha \beta) \gamma)\delta \ar@{-}[dr] &&&& \alpha(\beta(\gamma \delta)) \ar@{-}[dl] \\
& (\alpha(\beta \gamma))\delta \ar@{-}[rr] && \alpha((\beta \gamma)\delta). & } \]
In other words, we have a map $\mu_4 \colon S^1 \rightarrow \Map_\ast(X^4, X)$.  If $X$ is a loop space, one can see that this map is nullhomotopic.  Therefore, we say that $X$ is an $A_4$-\emph{space} if $\mu_4$ is nullhomotopic.

We can then continue, so an $A_5$-space has the map from the boundary of some 3-dimensional diagram to the space nullhomotopic, and we can define similarly in higher dimensions.  An $A_\infty$-space is an $A_n$-space for all $n$.

\begin{theorem}
If $X$ is a CW complex, then $X \simeq \Omega Y$ for some $Y$ if and only if:
\begin{enumerate}
\item $\pi_0 X$ is a group, and

\item $X$ is an $A_\infty$-space.
\end{enumerate}
\end{theorem}

\subsection{Work of Segal}

Assume that $X$ is a homotopy associative $H$-space (so that it is $A_3$).  Then $X$ is a monoid in $\Ho(\mathcal Top)$, the homotopy category of spaces.  Thus, the nerve of $X$ is a simplicial object in $\Ho(\mathcal Top)$, with $\nerve(X)_1=X$ and $\nerve(X)_2=X \times X$.  So, we get a functor $\Deltaop \rightarrow \Ho(\mathcal Top)$.

If a lift exists in the diagram
\[ \xymatrix{& \mathcal Top \ar[d] \\
\Deltaop \ar@{-->}[ur] \ar[r] & \Ho(\mathcal Top)} \] making it commutative up to natural isomorphism, then the space $X$ has a classifying space.  In other words, $X \simeq \Omega Y \simeq G$ with $BG \simeq Y$.

Consider the maps in ${\bf \Delta}$ of the form $\alpha^i \colon [1] \rightarrow [n]$ given by $0 \mapsto i-1$ and $1 \mapsto i$.

\begin{theorem}
Let $X$ be a CW complex.  Then $X \simeq \Omega Y$ if and only if
\begin{enumerate}
\item $\pi_0(X)$ is a group, and

\item there exists a simplicial space $A$ such that
\begin{enumroman}
\item $A_0=\ast$,

\item $A_1 \simeq X$, and

\item for every $n \geq 2$ the map $(\alpha^1)_\ast \times \cdots \times (\alpha^n)_\ast \colon A_n \rightarrow (A_1)^n$ is a homotopy equivalence.
\end{enumroman}
\end{enumerate}
In addition, $\Omega |A| \simeq X$ and $|A| \simeq Y$.
\end{theorem}

\section{Complete Segal spaces}

\centerline{(Talk given by Ilan Barnea)}

\vskip .1 in

Consider the category $\SSets^{\Deltaop}= \Sets^{\Deltaop \times \Deltaop}$ of simplicial spaces.  We can regard $\SSets$ as a subcategory of $\SSets^{\Deltaop}$, via the functor that takes a simplicial set $K$ to the constant simplicial space with $K$ at each level and all face and degeneracy maps the identity.

Recall that the category of simplicial spaces is monoidal under the cartesian product.  It also has a simplicial structure, where for a simplicial set $K$ and simplicial space $X$, we have $(K \times X)_n= K \times X_n$.  For simplicial spaces $X$ and $Y$ the simplicial set $\Map(X,Y)$ is given by the adjoint relation
\[ \Map(X,Y)_n=\Hom_{\SSets}(\Delta[n], \Map(X,Y))= \Hom_{\SSets^{\Deltaop}}(\Delta[n] \times X, Y). \]

There is another functor which takes a simplicial set $K$ to a simplicial space $F(K)$, given by $F(K)_n=K_n$.  Let $F(n)=F(\Delta[n])$.  Then $\Map(F(n),X) \cong X_n$.  We use this functor in the description of the cartesian closed structure on simplicial spaces, where
\[ (Y^X)_n= \Map(F(n), Y^X) \cong \Map(F(n) \times X, Y). \]

There is a model structure on the category of simplicial spaces where weak equivalences and cofibrations are levelwise, and fibrations have the left lifting property with respect to the maps which are both cofibrations and weak equivalences\cite[VIII, 2.4]{gj}. In this case, this model structure coincides with the Reedy model structure \cite{reedy}, \cite[15.8.7, 15.8.8]{hirsch}.

In fact, the category of simplicial spaces has the additional structure of a simplicial model category, so that if $X$ is cofibrant and $Y$ is fibrant, $\Map(X,Y)$ is the ``right" definition of mapping space, in that the mapping spaces are homotopy invariant.

Recall that if $f \colon X \rightarrow Y$ is a fibration and $Z$ is any cofibrant simplicial space, then $\Map(Z,X) \rightarrow \Map(Z,Y)$ is a fibration.  Letting $Z=F(n)$, we get that $X_n \rightarrow Y_n$ is a fibration for any $n$, and if $X$ is fibrant, it follows that each $X_n$ is a Kan complex.  Similarly, if $f \colon X \rightarrow Y$ is a cofibration and $Z$ is fibrant, then $\Map(Y,Z) \rightarrow \Map(X,Z)$ is a fibration.  Using the cofibration $F(0) \amalg F(0) \rightarrow F(1)$, we get that $d_1 \times d_0 \colon Z_1 \rightarrow Z_0 \times Z_0$ is a fibration.

\begin{definition} \cite{rezk}
A simplicial space $X$ is a \emph{Segal space} if it is Reedy fibrant and the Segal maps
\[ W_k \rightarrow \underbrace{W_1 \times_{W_0} \cdots \times_{W_0} W_1}_k \] are weak equivalences.
\end{definition}

Define the simplicial space $G(k) \subseteq F(k)$ by
\[ G(k)= \bigcup_{i=1}^k \alpha^i F(1) \subseteq F(k). \]  The Segal map in the previous definition is given by
\[ \Map(F(k), W) \rightarrow \Map(G(k),W). \]  By the above comments, this map is always a fibration; we want to require it to be a weak equivalence.  In particular, we have $\Map(F(k), W)_0 \rightarrow \Map(G(k),W)_0$ is surjective, so we can ``compose" 0-simplices of $W_1$.

For a Segal space $W$, define the \emph{objects} of $W$ to be $\ob(W)=W_{0,0}$.  Given $x,y \in \ob(W)$, define the mapping space between them to be the pullback
\[ \xymatrix{\map_W(x,y) \ar[r] \ar[d] & W_1 \ar[d] \\
\{(x,y)\} \ar[r] & W_0 \times W_0. } \]
Given an object $x$, define $\id_x= s_0(x) \in \map_W(x,x)$.  We say that $f,g \in \map_W(x,y)_0$ are \emph{homotopic} if $[f]=[g] \in \pi_0\map_W(x,y)$.

If $g \in \map_W(y,z)_0$ and $f \in \map_W(x,y)_0$, there exists $k \in W_{2,0}$ such that $k \mapsto (f,g)$ under the map $W_{2,0} \rightarrow W_{1,0} \times W_{1,0}$.  Then write $(g \circ f)_k=d_1(k)$, a \emph{composite} of $g$ and $f$ given by $k$.

We can define the homotopy category $\Ho(W)$ to have objects $W_{0,0}$ and morphisms given by
\[ \Hom_{\Ho(W)}(x,y)= \pi_0\map_W(x,y). \]
If $f \in \map_W(x,y)$ and $[f]$ is an isomorphism in $\Ho(W)$, then $f$ is a \emph{homotopy equivalence}.  Notice in particular that for any object $x$, $\id_x=s_0(x)$ is a homotopy equivalence.  Furthermore, if $[f]=[g] \in \pi_0\map_W(x,y)$, then $g$ is a homotopy equivalence if and only if $g$ is a homotopy equivalence.  Define $W_{heq} \subseteq W_1$ to be the components consisting of homotopy equivalences, and notice that $s_0 \colon W_0 \rightarrow W_{heq}$.

\begin{definition} \cite{rezk}
Let $W$ be a Segal space.  If $s_0 \colon W_0 \rightarrow W_{heq}$ is a weak equivalence, then $W$ is \emph{complete}.
\end{definition}

Let $\mathcal M$ be a model category and $T$ a set of maps in $\mathcal M$.  A $T$-\emph{local object} in $\mathcal M$ is a fibrant object $W$ in $\mathcal M$ such that $\Map(B,W) \rightarrow \Map(A,W)$ is a weak equivalence for every $f \colon A \rightarrow B$ in $T$.  A $T$-\emph{local equivalence} $g \colon X \rightarrow Y$ in $\mathcal M$ is a map such that $\Map(Y,W) \rightarrow \Map(X,W)$ is a weak equivalence for every $T$-local object $W$.

If $\mathcal M$ is a left proper and cellular model category \cite[13.1.1, 12.1.1]{hirsch}, then we can define a new model category structure $L_T \mathcal M$ on the same category with the same cofibrations, but with weak equivalences the $T$-local equivalences and fibrant objects the $T$-local objects.  If $f \colon X \rightarrow Y$ is a map with both $X$ and $Y$ $T$-local, then $f$ is a weak equivalence in $\mathcal M$ if and only if $f$ is a weak equivalence in $L_T \mathcal M$.  If $\mathcal M$ is a simplicial model category, then so is $L_T\mathcal M$ \cite[4.1.1]{hirsch}.

In the case of simplicial spaces, we first localize with respect to the maps $\{ \varphi_k \colon G(k) \rightarrow F(k) \}$ to get a model category whose fibrant objects are the Segal spaces.  Then consider the category $I$ with two objects and a single isomorphism between them.  Define $E=F(\nerve(I))$.  There is a map $F(1) \rightarrow E$ inducing $\Map(E,W) \rightarrow \Map(F(1), W)=W_1$ whose image is in $W_{heq}$.  In fact, $\Map(E,W) \rightarrow W_{heq}$ is a weak equivalence.  Using the map $E \rightarrow F(0)$, we get a map $W_0 \rightarrow \Map(E,W)$ which is a weak equivalence if and only if $s_0 \colon W_0 \rightarrow W_{heq}$ is a weak equivalence.  Thus, we want to localize with respect to the map $E \rightarrow F(0)$ to get a model category denoted $\css$ whose fibrant objects are the complete Segal spaces.  It follows that a map $f \colon X \rightarrow Y$ of complete Segal spaces is a weak equivalence in $\css$ if and only if $f$ is a levelwise weak equivalences of simplicial sets.

\section{Segal categories and comparisons}

We begin with the model structure on simplicial categories.  Recall that, given a simplicial category $\mathcal C$, we can apply $\pi_0$ to the mapping spaces to obtain an ordinary category $\pi_0 \mathcal C$, called the \emph{category of components} of $\mathcal C$.

\begin{theorem} \cite{simpcat}
There is a model structure on the category of small simplicial categories such that the weak equivalences are the simplicial functors $f \colon \mathcal C \rightarrow \mathcal D$ such that for any objects $x$ and $y$ the map $\Map_\mathcal C(x, y) \rightarrow \Map_\mathcal D(fx, fy)$ is a weak equivalence of simplicial sets, and the functor $\pi_0f \colon \pi_0 \mathcal C \rightarrow \mathcal D$ is an equivalence of categories.  The fibrant objects are the simplicial categories whose mapping spaces are Kan complexes.
\end{theorem}

These weak equivalences are often called \emph{Dwyer-Kan equivalences}, since they were first defined by Dwyer and Kan in their investigation of simplicial categories.

The idea behind Segal categories is that they are an ``intermediate" version between simplicial categories and complete Segal spaces.  They are like nerves of simplicial categories but with composition only defined up to homotopy.  They are Segal spaces but with a ``discreteness" condition in place of ``completeness."

\begin{definition}
A \emph{Segal precategory} is a simplicial space $X$ such that $X_0$ is a discrete simplicial set.  It is a \emph{Segal category} if the Segal maps
\[ \varphi_k \colon X_k \rightarrow \underbrace{X_1 \times_{X_0} \cdots \times_{X_0} X_1}_k \] are weak equivalences of simplicial sets for every $k \geq 2$.
\end{definition}

Let $\SSets^{\Deltaop}_{disc}$ denote the category of Segal precategories.  There is an inclusion functor
\[ I \colon \SSets^{\Deltaop}_{disc} \rightarrow \SSets^{\Deltaop}. \]

However, unlike for $\SSets^{\Deltaop}$, there is no model structure with levelwise weak equivalences and cofibrations monomorphisms.  For example, the map of doubly constant simplicial spaces $\Delta[0] \amalg \Delta[0] \rightarrow \Delta[0]$ cannot possibly be factored as a cofibration followed by an acyclic fibration.  Therefore, we cannot obtain our model structure for Segal categories by localizing such a model structure, as we did for Segal spaces.

However, given a Segal precategory, we can use the inclusion functor to think of it as an object in the Segal space model category and localize it to obtain a Segal space.  Generally, this procedure will not result in a Segal category, since it won't preserve discreteness of the 0-space.  However, one can define a modification of this localization that does result in a Segal space which is also a (Reedy fibrant) Segal category.  Given a Segal precategory $X$, we denote this ``localization" $LX$.  Notice that it can be defined functorially.

Since $LX$ is a Segal space, it has objects (given by $X_0$, since it is discrete) and mapping spaces $\map_{LX}(x,y)$ for any $x,y \in X_0$.  It also has a homotopy category $\Ho(LX)$.

\begin{definition}
A map $f \colon W \rightarrow Z$ of Segal spaces is a \emph{Dwyer-Kan equivalence} if:
\begin{enumerate}
\item $\map_W(x,y) \rightarrow \map_Z(fx,fy)$ is a weak equivalence of simplicial sets for any objects $x$ and $y$ of $W$, and

\item $\Ho(W) \rightarrow \Ho(Z)$ is an equivalence of categories.
\end{enumerate}
\end{definition}

\begin{prop} \cite{rezk}
A map $f \colon W \rightarrow Z$ of Segal spaces is a Dwyer-Kan equivalences if and only if it is a weak equivalence in $\css$.
\end{prop}

\begin{definition} \cite{rezk}
A map $f \colon X \rightarrow Y$ of Segal precategories is a \emph{Dwyer-Kan equivalence} if the induced map of Segal categories (Segal spaces) $LX \rightarrow LY$ is a Dwyer-Kan equivalence in the sense of the previous definition.
\end{definition}

\begin{theorem} \cite{thesis}, \cite{reedyfib}
There is a model structure $\Secat_c$ on the category of Segal precategories such that
\begin{enumerate}
\item the weak equivalences are the Dwyer-Kan equivalences,

\item the cofibrations are the monomorphisms (so that every object is cofibrant), and

\item the fibrant objects are the Reedy fibrant Segal categories
\end{enumerate}
\end{theorem}

We also want a model structure that looks more like a localization of the projective model structure.

\begin{theorem} \cite{thesis}, \cite{reedyfib}
There is a model structure $\Secat_f$ on the category of Segal precategories such that
\begin{enumerate}
\item the weak equivalences are the Dwyer-Kan equivalences, and

\item the fibrant objects are the projective fibrant Segal categories.
\end{enumerate}
\end{theorem}

We begin by looking at the comparison between Segal categories and complete Segal spaces.

\begin{prop} \cite{thesis}
The inclusion functor $I \colon \Secat_c \rightarrow \css$ has a right adjoint $R \colon \css \rightarrow \Secat_c$.
\end{prop}

We can define the functor $R$ as follows.  Suppose that $W$ is a simplicial space.  Define $RW$ to be the pullback
\[ \xymatrix{RW \ar[r] \ar[d] & \cosk_0(W_{0,0}) \ar[d] \\
W \ar[r] & \cosk_0(W_0).} \]  For example, at level 1 we get
\[ \xymatrix{(RW)_1 \ar[r] \ar[d] & W_{0,0} \times W_{0,0} \ar[d] \\
W_1 \ar[r] & W_0 \times W_0. } \]
If $W$ is a complete Segal space, then $RW$ is a Segal category; for example at level 2 we get
\[ \xymatrix{(RW)_1 \times_{(RW)_0} (RW)_1 \ar[r] \ar[d] & W_{0,0} \times W_{0,0} \times W_{0,0} \ar[d] \\
W_1 \times_{W_0} \times W_1 \ar[r] & W_0 \times W_0 \times W_0.} \]

\begin{theorem} \cite{thesis}
The adjoint pair $I \colon \Secat_c \rightleftarrows \css \colon R$ is a Quillen equivalence.
\end{theorem}

\begin{proof}[Idea of proof]
Showing that this adjoint pair is a Quillen pair is not hard, since cofibrations are exactly monomorphisms in each model category, and the weak equivalences in $\Secat_c$ are defined in terms of Dwyer-Kan equivalences in $\css$.  To prove that it is a Quillen equivalence, need to show that if $W$ is a complete Segal space, the map $RW \rightarrow W$ is a Dwyer-Kan equivalence of Segal spaces.
\end{proof}

Now we consider the comparison with simplicial categories.  We have the simplicial nerve functor $N \colon \mathcal{SC} \rightarrow \Secat_f$.

\begin{prop} \cite{thesis}
The simplicial nerve functor $N$ has a left adjoint $F \colon \Secat_f \rightarrow \mathcal{SC}$.
\end{prop}

\begin{proof}[Idea of proof]
Segal categories, at least fibrant ones, are local objects in the Segal space model structure (or here, the analogue of the Segal space model structure but where we localize the projective model structure on simplicial spaces, rather than the Reedy structure).  By definition, the Segal maps for a Segal category are weak equivalences.  Nerves of simplicial categories are ``strictly local" objects, in that the Segal maps are actually isomorphisms.  A more general result can be proved \cite{multisort} about the existence of a left adjoint functor to the inclusion functor of strictly local objects into the category of all objects, i.e., a rigidification functor.
\end{proof}

\begin{theorem} \cite{thesis}
The adjoint pair $N \colon \mathcal{SC} \leftrightarrows \Secat_f$ is a Quillen equivalence.
\end{theorem}

This result is much harder to prove than the previous comparison.  We need the model structure $\Secat_f$ here since it is more easily compared to $\mathcal{SC}$.  For example, in both cases the fibrant objects have mapping spaces which are Kan complexes, and not all monomorphisms are cofibrations.  The Quillen equivalence was proved first in the fixed object case, where methods of algebraic theories can be applied \cite{simpmon}.  then we can generalize to the more general case as stated in the theorem.

Hence, we get a chain of Quillen equivalences
\[ \mathcal{SC} \leftrightarrows \Secat_f \rightleftarrows \Secat_c \rightleftarrows \css. \]

\section{Quasi-categories}

\centerline{(Talk given by Yonatan Harpaz)}

\vskip .1 in

Consider topological spaces up to weak homotopy equivalence, i.e., $f \colon X \rightarrow Y$ which induces $\pi_n(X) \cong \pi_n(Y)$ for all $n$.  J.H.C.\ Whitehead proved:
\begin{enumerate}
\item Every space is weakly equivalent to a CW complex.

\item If two CW complexes are weakly homotopy equivalent, then they are actually homotopy equivalent.
\end{enumerate}

Between simplicial sets and topological spaces, there is a geometric realization functor $|-| \colon \SSets \rightarrow \mathcal Top$.

Recall in a model category that an object $X$ is \emph{fibrant} if the map $X \rightarrow \ast$ is a fibration.  It is \emph{cofibrant} if $\phi \rightarrow X$ is a cofibration.

In a model category, we want to define what we mean by ``$X \times I$" in order to define a ``homotopy" given by $X \times I \rightarrow Y$.

Define a map $S \rightarrow T$ of simplicial sets to be a \emph{weak equivalence} if its geometric realization is a weak homotopy equivalence of spaces.  In fact, we have an adjoint pair
\[ |-| \colon \SSets \leftrightarrows \mathcal Top \colon Sing. \]  Therefore, given a simplicial set $S$ we have a map $S \rightarrow Sing(|S|)$, and given a topological space $X$ we have a map $|Sing(X)| \rightarrow X$, and these are weak equivalences for $S$ and $T$ sufficiently nice.  The left adjoint preserves (acyclic) cofibrations, and the right adjoint preserves (acyclic) fibrations.  The consequence of this adjointness is that Kan complexes, which can be written as $Sing(|S|)$ for some $S$, are good models for spaces.

In the world of topological categories, we say that a functor $f \colon \mathcal C \rightarrow \mathcal D$ is a \emph{weak equivalence} if it induces
\begin{enumerate}
\item $\Hom(X,Y) \rightarrow \Hom(fX, fY)$ is a weak homotopy equivalence, and

\item an equivalence on $\pi_0$.
\end{enumerate}
We can model topological categories by simplicial categories.  The fibrant simplicial categories have mapping spaces Kan complexes, it is less clear how to describe cofibrant objects.

We want to define another model: quasi-categories, which fit into a chain of adjoints
\[ \SSets \leftrightarrows \mathcal{SC}at \leftrightarrows \mathcal Top \mathcal Cat. \]  We consider the adjoint pair
\[ C^\Delta \colon \SSets \leftrightarrows \mathcal{SC}at \colon N^\Delta. \]

We first consider the discrete version: $C \colon \SSets \leftrightarrows \mathcal Cat \colon \nerve$.  Here we have $\ob(C(\Delta^n))=\{0, \cdots, n\}$ and $\Hom(i,j)=\bullet$ for $i \leq j$ and is empty otherwise.  The nerve is given by $\nerve(\mathcal C)_n= \Hom_{\mathcal Cat}(C(\Delta^n), \mathcal C)$.

In the simplicial case, we still have $\ob(C^\Delta(\Delta^n))=\{0, \ldots, n\}$, but for $i \leq j$ we now have
\[ \Hom(i,j)=(\Delta^1)^{\{i+1, \ldots, j-1\}}. \]  (Note that $(\Delta^1)^\phi = \ast$.)  Composition is given by
\[ (\Delta^1)^{\{i+1, \ldots, j-1\}} \times (\Delta^1)^{\{j+1, \ldots, k-1\}} \rightarrow (\Delta^1)^{\{i+1, \ldots, k+1\}}. \]  Similarly to before, we have $N^\Delta(\mathcal C) = \Hom_{\mathcal{SC}at}(C^\Delta(\Delta^n), \mathcal C)$.  This functor $N^\Delta$ is called the \emph{coherent nerve}.

Using these functors, we define a new model structure on the category $\SSets$.  The new version of $S \times I$ will use $I=\nerve(\bullet \rightleftarrows \bullet)=:E$.  The fibrant objects will be the simplicial sets which look like coherent nerves of fibrant simplicial categories, or quasi-categories, as given in Definition \ref{quasi}.

Suppose that $\mathcal C$ is a fibrant simplicial category.  Then there is a lift
\[ \xymatrix{V[n,k] \ar[r] \ar[d] & N^\Delta(\mathcal C) \\
\Delta^n \ar@{-->}[ur] & } \] if and only if there is a lift
\[ \xymatrix{C^\Delta(\Lambda^n_k) \ar[r] \ar[d] & \mathcal C \\
C^\Delta(\Delta^n) \ar@{-->}[ur] & .}\]

What is $C^\Delta(\Delta^n)$?  It will have the same objects as $C^\Delta(\Delta^n)$, since $V[n,k]$ has the same vertices as $\Delta^n$.  Furthermore, $\Hom(i,j)$ will be the same when $(i,j) \neq (0,n)$, but will change in this one case.  For example, when $n=2$, we get a new element in $\Hom(0,2)$.  But, the change is given by an acyclic cofibration of simplicial sets.

So, to get a lift, we know where to send objects, but we just have one mapping space missing.  But, this still works since these are Kan complexes.

\section{Comparison on quasi-categories with other models}

\centerline{(Talk given by Tomer Schlank)}

\vskip .1 in

Kan complexes have the property that
\[ |\Map_{\SSets}(X,Y)| = \Map_{\mathcal Top}(|X|, |Y|). \]

Having a lift in any diagram of the form
\[ \xymatrix{V[n,k] \ar[r] \ar[d] & X \\
\Delta^n \ar@{-->}[ur] & } \] for $0 \leq k \leq n$ encodes spaces.  If we consider only $0<k<n$ and require the lift to be unique, then such diagrams encode categories.  Combining the two, taking the existence of a lift for $0<k<n$, encodes simplicial categories.

We want a model structure on the category of simplicial sets whose fibrant objects are precisely the quasi-categories.  So, we want each $V[n,k] \rightarrow \Delta^n$ to be an acyclic cofibration.  On the other hand, when $k=0$ or $k=n$, we don't want to require it to be a weak equivalence.

\begin{definition}
If $X$ and $Y$ are simplicial sets, then a map $f \colon X \rightarrow Y$ is a \emph{categorical equivalence} if $C^\Delta(X) \rightarrow C^\Delta(Y)$ is a weak equivalence of simplicial categories.
\end{definition}

Recall that maps between model categories are pairs of functors
\[ F \colon \mathcal M_1 \rightleftarrows \mathcal M_2 \colon G \] which are adjoint, so that $\Hom(FX,Y) \cong \Hom(X,GY)$.  Furthermore, we require that $F$ preserves cofibrations and acyclic cofibrations, which is true if and only if $G$ preserves fibrations and acyclic fibrations.  Such a Quillen pair induces a map on homotopy categories $\Ho(\mathcal M_1) \rightarrow \Ho(\mathcal M_2)$ given by $X \mapsto F(X^{cof})$.  If this map is an equivalence, we have a Quillen equivalence.  Equivalently, if $X$ is fibrant in $\mathcal M_1$ and $Y$ is fibrant in $\mathcal M_2$, then the isomorphism $\Hom(FX,Y) \leftrightarrow \Hom(X,GY)$ takes weak equivalences to weak equivalences. We want the adjoint pair $C^\Delta \colon \Sets^{\Deltaop} \leftrightarrows \mathcal{SC}at \colon N^\Delta$ to be a Quillen equivalence.

Suppose we have a Quillen pair $F \colon \mathcal C \leftrightarrows \mathcal D \colon G$ with $\mathcal C$ a left proper, combinatorial model category with weak equivalences preserved under filtered colimits. Let $CF$ be a class of morphisms (in fact, cofibrations) in $\mathcal D$, and assume that $F(CF) \subseteq \text{cof}(\mathcal C)$.  Assume that $F(CF^\perp) \subseteq \text{we}(\mathcal C)$, where $CF^\perp$ denotes the class of maps with the right lifting property with respect to the maps in $CF$.  This gives a left proper, combinatorial model structure on $\mathcal D$ so that $(F,G)$ is a Quillen pair.  This method gives a model structure on $\SSets^{\Deltaop}$.

To show that $(C^\Delta, N^\Delta)$ is a Quillen equivalence, we need to show that if $X$ is (cofibrant) simplicial set and $Y$ is a fibrant simplicial category, then $C^\Delta(X) \rightarrow Y$ is a weak equivalence if and only if $X \rightarrow N^\Delta(Y)$ is a weak equivalence, if and only if $C^\Delta(X) \rightarrow C^\Delta(N^\Delta(Y))$ is a weak equivalence.  The idea is to compose this last map with the counit $C^\Delta (N^\Delta(Y)) \rightarrow Y$ and show that the counit is a weak equivalence.

\section{Applications: Homotopy-theoretic constructions and derived Hall algebras}

Having model categories whose objects are ``homotopy theories" enables us to understand the ``homotopy theory of homotopy theories."  We can compare to ad-hoc constructions already developed for model categories and verify them, then make new ones.  An example is a definition of the ``homotopy fiber product" of model categories.

\begin{definition}
Let
\[ \xymatrix@1{\mathcal M_1 \ar[r]^{F_1} & \mathcal M_3 & \mathcal M_2 \ar[l]_{F_2} } \] be a diagram of left Quillen functors of model categories.  Define their \emph{weak homotopy fiber product} $\mathcal M_w$ to have objects 5-tuples $(x_1, x_2, x_3; u,v)$ where each $x_i$ is an object of $\mathcal M_1$ and $u$ and $v$ are maps in $\mathcal M_3$ given by
\[ \xymatrix@1{F_1(x_1) \ar[r]^u & x_3 & F_2(x_2). \ar[l]_v } \]  Morphisms
\[ (x_1, x_2, x_3; u,v) \rightarrow (y_1, y_2, y_3; w,z) \] are given by triples $(f_1, f_2, f_3)$ where $f_i \colon x_i \rightarrow y_i$ and the diagram
\[ \xymatrix{F_1(x_1) \ar[r]^u \ar[d]^{F_1(f_1)} & x_3 \ar[d]^{f_3} & F_2(x_2) \ar[l]_v \ar[d]^{F_2(f_2)} \\
F_1(y_1) \ar[r]^w & y_3 & F_3(x_3) \ar[l]_z } \] commutes.  The \emph{homotopy fiber product} $\mathcal M= \mathcal M_1 \times_{\mathcal M_3} \mathcal M_2$ is the full subcategory of $\mathcal M_w$ whose objects satisfy the condition that $u$ and $v$ are weak equivalences in $\mathcal M_3$.
\end{definition}

There is a model structure on the category $\mathcal M_w$ in which weak equivalences and cofibrations are levelwise.  The subcategory $\mathcal M$ does not have a model structure, as the additional property on $u$ and $v$ is not preserved by all limits and colimits.  In some cases, however, we can localize the model structure on $\mathcal M_w$ so that the fibrant objects are in $\mathcal M$.  Otherwise, we can just think of $\mathcal M$ as a category with weak equivalences.

We want to consider this construction in the context of the homotopy theory of homotopy theories, in particular in the setting of complete Segal spaces.  To translate, we have two choices, using the functor $L_C$ that takes a model category to a complete Segal space.  We can either first apply $L_C$ to the diagram and then take the homotopy pullback of complete Segal spaces $L_C\mathcal M_1 \times{L_C \mathcal M_3} L_C \mathcal M_2$, or we can take the complete Segal space associated to the homotopy fiber product, $L_C\mathcal M$.

\begin{prop} \cite{fiberprod}
The complete Segal spaces $L_C\mathcal M$ and $L_C\mathcal M_1 \times_{L_C \mathcal M_3} L_C \mathcal M_2$ are weakly equivalent.
\end{prop}

This result verifies that the definition of the homotopy fiber product of model categories is in fact the correct one, from a homotopy-theoretic point of view.

We now consider an application of this construction, that of derived Hall algebras.  Let $\mathcal A$ be an abelian category with $\Hom(X,Y)$ and $\Ext^1(X,Y)$ finite for any objects $X$ and $Y$ of $\mathcal A$.

\begin{definition}
The \emph{Hall algebra} $\mathcal H(\mathcal A)$ associated to $\mathcal A$ is generated as a vector space by isomorphism classes of objects in $\mathcal A$ and has multiplication given by
\[ [X] \cdotp [Y]= \sum_{[Z]} g^Z_{X,Y} [Z] \] where the \emph{Hall numbers} are
\[ g^Z_{X,Y} = \frac{|\{0 \rightarrow X \rightarrow Z \rightarrow Y \rightarrow 0\}|}{|\Aut(X)||\Aut(Y)|}. \]
\end{definition}

As a motivating example, let $\mathfrak g$ be a Lie algebra of type $A$, $D$, or $E$.  It has an associated Dynkin diagram, for example $A_3$, which looks like $\bullet - \bullet - \bullet$.  Take a quiver $Q$ on this diagram, for example $\bullet \rightarrow \bullet \leftarrow \bullet$.  Let $\mathcal A =Rep(Q)$, the category of $\mathbb F_q$-representations of $Q$ for some finite field $\mathbb F_q$.  We get an associated Hall algebra $\mathcal H(Rep(Q))$.

However, associated to $\mathfrak g$ we also have the quantum enveloping algebra
\[ U_q(\mathfrak g)= U_q(\mathfrak n^+) \otimes U_q(\mathfrak h) \otimes U_q(\mathfrak n^-). \]  We write $U_q(\mathfrak b)= U_q(\mathfrak n^+) \otimes U_q(\mathfrak n^-)$.   Ringel proved that $U_q(\mathfrak b)$ is closely related to $\mathcal H(Rep(Q))$.  A natural question is therefore whether there is an algebra arising from $\mathcal A$ corresponding to $U_q(\mathfrak g)$ in this same way.  The conjecture is that the ``right" algebra is obtained from the root category $\mathcal D^b(\mathcal A)/T^2$, where $D^b(\mathcal A)$ is the bounded derived category, which is a triangulated category, and $T$ is its shift functor \cite{px}.

An immediate problem in attempting to find such an algebra is that the root category is triangulated rather than abelian.  Thus, the question becomes how to get a ``Hall algebra" from a triangulated category.  To\"en has developed an approach as follows.

Let $Ch(\mathbb F_q)$ be the category of chain complexes over $\mathbb F_q$.  Let $\mathcal T$ be a small dg category over $\mathbb F_q$ (or category enriched over $Ch(\mathbb F_q)$) such that for any objects $x$ and $y$ of $\mathcal T$, the complex $\mathcal T(x,y)$ is cohomologically bounded and has finite-dimensional cohomology groups.  Define $\mathcal M(\mathcal T)$ to be the category of $\mathcal T^{op}$-modules, or dg functors $\mathcal T \rightarrow Ch(\mathbb F_q)$.  This category has a model structure given by levelwise weak equivalences and fibrations, and furthermore its homotopy category $\Ho(\mathcal M(\mathcal T))$ is triangulated.   There is also a model category $\mathcal M(\mathcal T)^{[1]}$ whose objects are the morphisms in $\mathcal M(\mathcal T)$.

We have a diagram of left Quillen functors
\[ \xymatrix{\mathcal M(\mathcal T)^{[1]} \ar[r]^t \ar[d]^{s \times c} & \mathcal M(\mathcal T) \\
\mathcal M(\mathcal T) \times \mathcal M(\mathcal T) &} \] where $s$ is the source map, $t$ is the target map, and $c$ is the cone map, given by $c(x \rightarrow y)= y \amalg_x 0$, where 0 is the zero object of $\mathcal M(\mathcal T)$.

Let $\mathcal P(\mathcal T)$ be the full subcategory of $\mathcal M(\mathcal T)$ of perfect objects, which in this case consists of the dg functors whose images in $Ch(\mathbb F_q)$ are cohomologically bounded with finite-dimensional cohomology groups.  We restrict the above diagram to subcategories of weak equivalences between perfect and cofibrant objects:
\[ \xymatrix{w(\mathcal P(\mathcal T)^{[1]})^{cof} \ar[r]^t \ar[d]^{s \times c} & w\mathcal P(\mathcal T)^{cof} \\
w\mathcal P(\mathcal T)^{cof} \times w\mathcal P(\mathcal T)^{cof} &.} \]  We apply the nerve functor to this diagram and denote it as follows:
\[ \xymatrix{X^{(1)} \ar[r]^t \ar[d]^{s \times c} & X^{(0)} \\
X^{(0)} \times X^{(0)} & . } \]

Define the \emph{derived Hall algebra} $\mathcal{DH}(\mathcal T)$ as a vector space by $\mathbb Q_c(X^{(0)})$, the finitely supported functions $\pi_0(X^{(0)}) \rightarrow \mathbb Q$, and define the multiplication $\mu = t_! \circ (s \times c)^*$.  While $(s \times c)^*$ is the usual pullback, $t_!$ is a more complicated push-forward.

\begin{theorem} \cite{toendha}
\begin{enumerate}
\item $\mathcal{DH}(\mathcal T)$ is an associative, unital algebra.

\item The multiplication is given by $[x] \cdotp [y]= \sum_{[z]}g^z_{x,y}[z]$, where the \emph{derived Hall numbers} are
\[ g^z_{x,y} = \frac{|[x,z]_y| \cdotp \prod_{i>0}|x,z[-i]|^{(-1)^i}}{|\Aut(x)| \cdotp \prod_{i>0}|[x,x[-i]]|^{(-1)^i}} \] where brackets denote maps the homotopy category and $[x,z]_y$ denotes the maps $x\rightarrow z$ with cone $y$.

\item The algebra $\mathcal{DH}(\mathcal T)$ depends only on $\Ho(\mathcal M(\mathcal T))$.
\end{enumerate}
\end{theorem}

The homotopy fiber product of model categories is the key tool in the proof that the derived Hall algebra is associative.

\begin{theorem} \cite{dhacss}
To\"en's derived Hall algebra construction can be given for a suitably finitary stable complete Segal space.
\end{theorem}

By ``stable" here we mean that the homotopy category is triangulated.  Xiao and Xu prove that his formula (2) holds for any finitary triangulated category, but the additional homotopy-theoretic information may be useful.  Unfortunately, in either case, the root category is not ``finitary" so we cannot yet find a derived Hall algebra associated to it.

\section{Commutative algebra for quasi-categories}

\centerline{(Two talks given by Vladimir Hinich, plus some additional notes)}

\vskip .1 in

Let $k$ be a field of characteristic zero and $A$ an associative algebra over $k$.  Let $P$ be the free bimodule resolution of $A$,
\[ \xymatrix@1{\cdots \ar[r]^-d & A \otimes A \otimes A \ar[r]^-d & A \otimes A \ar[r] & A \ar[r] & 0 }\] where
\[ d(a_1 \otimes \cdots \otimes a_n)=\sum_i a_1 \otimes \cdots \otimes a_i a_{i+1} \otimes \cdots \otimes a_n. \]  We define
\[ \Hoch^*(A)=\Hom_{A \otimes A^{op}}(P,A). \]
In other words,
\[ \Hoch^n(A)=\Hom(A^{\otimes n}, A) \] for all $n \geq 0$.  The differential
\[ d \colon \Hoch^n(A) \rightarrow \Hoch^{n+1}(A) \] is given by the formula
\[ \begin{aligned}
(df)(a_1, \ldots, a_{n+1})& =a_1f(a_2, \ldots, a_{n+1}) \\ 
& = \sum_{i=1}^n(-1)^i f(a_1, \ldots, a_ia_{i+1}, \ldots, a_{n+1})-(-1)^nf(a_1, \ldots, a_n)a_{n+1} 
\end{aligned} \] for $f \in \Hom(A^{\otimes n},A)$.
The cohomology of this complex is $HH^*(A)$, the \emph{Hochschild cohomology} of $A$.

In 1963, Gerstenhaber proved that $HH^*(A)$ is commutative and $HH^*(A)[1]$ is a Lie algebra.  The two structures are compatible in the sense that an analogue of the Leibniz rule holds.  This structure is called a \emph{Gerstenhaber algebra} and is given by a Gerstenhaber operad $e_2$.  In the 1970s, F.\ Cohen considered the little cubes operad $\mathcal E(2)$ and showed that $H_*(\mathcal E(2))=e_2$.  In other words, for each $n$, $H_*(\mathcal E(2)(n))=e_2(n)$.

In 1993, Deligne asked whether this isomorphism could be lifted to an action of the $\mathcal E(2)$-operad on $\Hoch^*(A,A)$.  This conjecture was proved in the late 1990s by many different people.  The idea of the proof, as developed by Lurie in \cite{dag6}, is as follows.
\begin{enumerate}
\item The little cubes operad $\mathcal E(2)$ acts on a double loop space $\Omega^2(X)$.

\item Let $\mathcal C$ be some kind of higher category and $X$ an object of $\mathcal C$.  Then $\End(X)$ is an associative algebra (for a category) and $\End(\id_X)$ is a double loop space.

\item Look at the category with objects associative algebras over $k$ and, for any such algebras $A$ and $B$, $\Map(A,B)$ is the category of $A-B$-bimodules, with composition given by
  \[ (_AM_B, _BN_C) \mapsto M \otimes_BN. \]  The identity map is given by the bimodule $_AA_A$.  The endomorphisms of the identity map should be described by $R\Hom_{A \otimes A^{op}}(A,A)= Hoch^*(A)$.
\end{enumerate}
However, Lurie does not specify what the ``higher category" is; it is not an $(\infty, 1)$-category.

These steps given a ``moral explanation" for the Deligne Conjecture.  They acquire a precise sense in the setup of quasi-categories, and this talk is based on excerpts from two of Lurie's papers \cite{dag3}, \cite{dag6}.  In what follows, we use quasi-categorical versions of symmetric monoidal categories and of operads.  To go further into this proof, we need to understand quasi-operads and symmetric monoidal categories.

The notion of symmetric monoidal category is not very simple, even in the classical setup because of the necessity of describing the axioms of associativity ($(x \otimes y) \otimes z \cong x \otimes (y \otimes z)$) and commutativity ($x \otimes y \cong y \otimes x$).  Therefore, we start our discussion presenting a further version of the definition which itself takes care of these constraints.  This definition will be further generalized to the world of quasi-categories.

Back in the world of sets, recall the definition of a colored operad $\mathcal O$.  It has a collection of colors $[\mathcal O]$, together with $a \colon I \rightarrow [\mathcal O]$ for a finite set $I$ (the multiple inputs) and $b \in [\mathcal O]$ (the one output).  Then $\mathcal O(a,b)$ is a set, and composition is given by
\[ \mathcal O(b,c) \times \prod_{j \in J} \mathcal O(a|_{\varphi^{-1}(j)}, b(j)) \rightarrow \mathcal O(a,c) \] where we have the diagram
\[ \xymatrix{I \ar[r]^a \ar[d]^\varphi & [\mathcal O] \\
J \ar[r] \ar[ur]^b & \ast \ar[u]_c} \] satisfying associativity.

A useful exercise is to check that this definition includes the action of the symmetric groups automatically, where possible, as well as equivariance of composition.  The same kind of definition works for topological operads; in this case $\mathcal O(a,b)$ is a topological space.

Another name for colored operads is that of \emph{pseudo-tensor categories}, due to the fact that the unary operations $\mathcal O(a,b)$ with $|I|=1$ endow $\mathcal O$ with the structure of a category, whereas the other operations, if representable, define a symmetric monoidal structure.

A colored operad $\mathcal O$ has an associated category $\mathcal O_1$ with objects $[\mathcal O]$ and $\mathcal O_1(a,b)= \mathcal O(a,b)$.  Notice that a category is a special case of a colored operad, where we have only unary operations.

For any function $a \colon I \rightarrow [\mathcal O]$, we have a functor $\mathcal O_1 \rightarrow \Sets$ given by $b \mapsto \mathcal O(a,b)$.  We assume that all these functors are representable and choose the representing object $\otimes^Ia$, defined up to unique isomorphism.  It contains precisely the data of a symmetric monoidal category.

\begin{definition}
A \emph{symmetric monoidal category} is a colored operad for which all functors above are representable.
\end{definition}

\begin{example}
The terminal operad in $\Sets$ is given by $\mathcal O(a,b)= \ast$.
\end{example}

It is convenient to reformulate the notion of colored operad, as well as that of a symmetric monoidal category, so that all compositions, not just the unary operations, are part of a category structure.  This process can be done as follows.

Let $\mathcal O$ be a colored operad, and let $\Gamma$ be the category of finite pointed sets \cite{segal}.  Define a category $\mathcal O^\otimes$, together with a functor $\mathcal O^\otimes \rightarrow \Gamma$, with
\[ \ob(\mathcal O^\otimes)=\{a \colon I \rightarrow [\mathcal O]\} \] and
\[ \Hom_{\mathcal O^\otimes}(a,b)=\coprod_{\varphi \colon I_* \rightarrow J_*} \prod_j \mathcal O(a|_{\varphi^{-1}(j)}, b(j)). \] Here $I_*=I \amalg \{\ast\}$ are the objects of $\Gamma$ and $\varphi \colon I_* \rightarrow J_*$ is a map from $I$ to $J$.    Let $\langle n \rangle = \{ 1, \ldots, n\}_*$ and notice that $\mathcal O^\otimes_{\langle n \rangle} = (\mathcal O^\otimes_{\langle 1 \rangle})^n$.  The map $\mathcal O^\otimes \rightarrow \Gamma$ satisfies some properties, which will be written down precisely below in the context of quasi-categories.

\begin{definition} \cite{lurie}
A map $\varphi \colon \langle m \rangle \rightarrow \langle n \rangle$ is \emph{inert} if for every $i =1, \ldots, n$, $|\varphi^{-1}(i)|=1$.  It is \emph{semi-inert} if $|\varphi^{-1}(i)|<1$ for each $i$.
\end{definition}

\begin{definition} \cite{lurie}
Let $p \colon \mathcal C \rightarrow \mathcal D$ be a map of quasi-categories, and $\alpha \colon c \rightarrow c'$ a morphism in $\mathcal C$.  Then $\alpha$ is $p$-\emph{cocartesian} if the diagram
\[ \xymatrix{\Map_{\mathcal C}(c',c'') \ar[r] \ar[d] & \Map_\mathcal C(c, c'') \ar[d] \\
\Map_\mathcal D(pc', pc'') \ar[r] & \Map_\mathcal D(pc,pc'')} \] is homotopy cartesian.
\end{definition}

Note that the horizontal arrows are not canonically defined; however, the definition can be made precise.

\begin{definition}
Let $\mathcal O^\otimes$ be a quasi-category.  A map $p \colon \mathcal O^\otimes \rightarrow N(\Gamma)$ is a \emph{quasi-operad} if:
\begin{enumerate}
\item For any inert $\varphi \colon \langle m \rangle \rightarrow \langle n \rangle$ and any $x \in \mathcal O^\otimes_{\langle m\rangle}$, there exists $\widetilde \varphi \colon x \rightarrow y$ over $\varphi$ which is $p$-cocartesian; alternatively, for any $\varphi$ in $N(\Gamma)_1$, there is a lifting (or cocartesian fibration) $\varphi_! \colon \mathcal O^\otimes_{\langle m \rangle} \rightarrow \mathcal O^\otimes_{\langle n \rangle}$.  (In other words, for any arrow in $N(\Gamma)$ whose source lifts, the arrow can be lifted to a $p$-cocartesian arrow.)

\item If $\rho_i \colon \langle n \rangle \rightarrow \langle 1 \rangle$ is an inert map with $\rho_i(i)=1$, then
\[ \prod_i(\rho_i)_! \colon \mathcal O^\otimes_{\langle n \rangle} \rightarrow \prod_i \mathcal O^\otimes_{\langle 1 \rangle} \] is a categorical equivalence.

\item For every $\varphi \colon \langle m \rangle \rightarrow \langle n \rangle$, $c \in \mathcal O^\otimes_{\langle  m \rangle}$, and $c' \in \mathcal O^\otimes_{\langle n \rangle}$, there is an equivalence of maps in the fiber
    \[ \Map^\varphi(c,c') \rightarrow \prod\Map^{\rho_i(\varphi)}(c_!, (\rho_i)_!(c')). \]
\end{enumerate}
\end{definition}

\begin{example}
If $\mathcal O$ is a topological operad, then $Sing(\mathcal O)$ is an operad in $\SSets$ (actually in Kan complexes).  Then we get a simplicial category $Sing(\mathcal O)^\otimes$ and get a quasi-operad $N^\Delta(Sing(\mathcal O)^\otimes) \rightarrow \nerve(\Gamma)$.
\end{example}

\begin{example}
Let $\widetilde{\mathcal O}$ be a topological colored operad.  We can describe it by a functor of topological categories $\widetilde {\mathcal O}^\otimes \rightarrow \Gamma$.  Passing to singular simplices and to the nerve, we get a map $p \colon \mathcal O^\otimes \rightarrow \nerve(\Gamma)$ of quasi-categories.  Thus, any topological operad gives rise to a quasi-operad.  In particular, all operads in $\Sets$ give rise to quasi-operads.
\end{example}

\begin{example}
\begin{enumerate}
\item For the commutative operad $Com$, we get
\[ Com^\otimes=\nerve(\Gamma) \rightarrow \nerve(\Gamma), \] the identity map.

\item For the associative operad $Assoc$, we get $\ob(Assoc^\otimes)=\ob(\Gamma)$ and
\[ \Hom_{Assoc^\otimes}(\langle m \rangle, \langle n \rangle)= \{(\varphi \colon \langle m \rangle \rightarrow \langle n \rangle, \leq) \} \] where the maps $\varphi$ are in $\Gamma$ and $\leq$ denotes a total order on all preimages.  Alternatively, we could write this set as the collection of all pairs $(\varphi, \omega)$ where $\varphi$ is a morphism of $\Gamma$ and $\omega$ is a collection of total orders on $\varphi^{-1}(j)$ for each $j=1, \ldots, n$.  We will denote the corresponding operad
\[ p \colon Assoc^\otimes \rightarrow \nerve(\Gamma). \]

\item We can take the trivial operad $Triv$, and then $Triv^\otimes \subseteq \nerve(\Gamma)$ consists of the inert arrows.

\item The sub-quasi-category of semi-inert arrows gives $E[0]$, whose algebras are pointed objects.  (Classically, $\mathbb E[0]$-algebras are spaces with a marked point.)

\item More generally, $E[k]$ comes from the topological operad $\mathcal E(k)$ of little $k$-cubes.  It has objects $\{\langle n \rangle \}$.
\end{enumerate}
\end{example}

\begin{definition}
A \emph{symmetric monoidal quasi-category} is a quasi-operad $p \colon \mathcal C^\otimes \rightarrow \nerve(\Gamma)$ such that $p$ is a cocartesian fibration.
\end{definition}

We now look again at the motivating problem.  The Gerstenhaber operad appears in topology as the homology of a certain topological operad (the small cubes operad) in dimension 2.

\begin{definition}
Fix $k=0, 1, \ldots$ and let $\square^k$ be the standard $k$-cube.  Define $\mathcal E[k](n)=Rect(\square^k \times \{1, \ldots, n\}, \square^k)$, the space of rectilinear embeddings.  The latter means that the images of different copies of $\square^k$ have no intersection, and each $\square^k \rightarrow \square^k$ is the product of rectilinear maps in each coordinate.  The collection
\[ \{ \mathcal E[k](n) \mid n \geq 0\} \] forms an operad in $\mathcal Top$, the \emph{small} $k$-\emph{cubes operad}.
\end{definition}

\begin{theorem}
$Gerst(n)=H_*(\mathcal E[k](n))$.
\end{theorem}

Now the following result seems very natural.

\begin{theorem}[Deligne Conjecture]
There is a natural $\mathcal E[2]$-algebra structure on $\Hoch^*(A)$.
\end{theorem}

This conjecture has been proved in different formulations by different authors around the '90s, but there is still a feeling that there is no full understanding of what is going on.

To understand the problem more fully, we investigate tensor products of operads.  Let $\mathcal O_1$ and $\mathcal O_2$ be (conventional) topological operads.  Maps $a \colon X^I \rightarrow X$ and $b \colon X^J \rightarrow X$ commute if the diagram
\[ \xymatrix{X^{I \times J} \ar[r]^{b^I} \ar[d]_{a^J} & X^I \ar[d] \\
X^J \ar[r] & X} \] commutes.

\begin{definition}
An $(\mathcal O_1, \mathcal O_2)$-algebra is a topological space with commuting structures of $\mathcal O_1$- and $\mathcal O_2$-algebras.
\end{definition}

May claimed in the '80s that such a pair of commuting structures is given by a canonically-defined structure of an algebra over a new operad $\mathcal O_1 \otimes^M \mathcal O_2$, the May tensor product.

It is convenient to present quasi-operads as fibrant objects in a certain simplicial model category.  We first describe the objects of the underlying category.

\begin{definition} \cite{dag3}
A \emph{marked simplicial set} is a pair $(X, \Sigma)$ where $X$ is a simplicial set and $\Sigma$ is a collection of edges containing the degenerate ones.
\end{definition}

\begin{definition} \cite{dag3}
A \emph{preoperad} is a map of marked simplicial sets
\[ p \colon (X, \mathcal E \subseteq X_1) \rightarrow (\nerve(\Gamma, \text{inert}))\]
where $s_0(X_0) \subseteq \mathcal E$ and $p(\mathcal E)$ is inert. \end{definition}

\begin{theorem} \cite{dag3}
The category of preoperads admits a simplicial model category structure for which:
\begin{enumerate}
\item the simplicial structure is defined as in $\SSets_\ast$,

\item cofibrations are the injective maps, and

\item weak equivalences are those giving rise to homotopy equivalences of simplicial homs into each preoperad $p \colon (X, \Sigma) \rightarrow \nerve(\Gamma)$ where each $X$ is an $\infty$-operad and $\Sigma$ is the class of inert arrows.
\end{enumerate}
\end{theorem}

For $p \colon \mathcal O^\otimes \rightarrow N(\Gamma)$, denote by $\mathcal O^{\otimes, \natural}$ the preoperad $(X, \Sigma)$ with $\Sigma$ the inert arrows.

Let $K^\# =(K, K_1)$ be the maximal pair and $K^\flat=(K, s_0(K_0))$ be the minimal pair.  If $\overline X=(X, \mathcal E)$, then \[ \Hom(K, \Map^\#(\overline X, \overline Y))= \Hom_{N(\Gamma)}(\overline X \times K^\#, \overline Y) \] which is a Kan complex, and
\[ \Hom(K, \Map^\flat(\overline X, \overline Y))= \Hom_{N(\Gamma)}(\overline X \times K^\flat, \overline Y) \] which is a quasi-category in good cases.  The cofibrations are injections, and fibrant objects are the marked simplicial sets $(X, \mathcal E)$ where $X$ is a quasi-operad and $\mathcal E$ consists of the cocartesian lifts of inert maps.  A map $f \colon X \rightarrow Y$ is a weak equivalence if for any $\mathcal O$ fibrant, $\Map^\#(Y, \mathcal O) \rightarrow \Map^\#(X, \mathcal O)$ is a weak equivalence.

Let $\overline X$ and $\overline Y$ be preoperads, and consider the join $\wedge \colon \nerve(\Gamma) \times \nerve(\Gamma) \rightarrow \nerve(\Gamma)$, defined by $((I,i), (J,j)) \mapsto (I \times J)/((i \times J) \amalg (I \times j))$.  Write $\overline X \odot \overline Y$ for the composite
\[ \overline X \times \overline Y \rightarrow \nerve(\Gamma) \times \nerve(\Gamma) \rightarrow \nerve(\Gamma). \]

If $(X, \Sigma)$ and $(X', \Sigma')$ are operads (fibrant preoperads), the product $(X, \Sigma) \odot (X', \Sigma')$ is not usually fibrant.

We can use fibrant replacement to define Lurie's tensor product.

\begin{definition}
If $\mathcal O$ and $(\mathcal O')^\otimes$ are operads, then $\mathcal O^\otimes \otimes^L (\mathcal O')^\otimes$ is a fibrant replacement of $(\mathcal O \odot \mathcal O')^\otimes$.
\end{definition}

Thus, we have the following.

\begin{definition}
Let $\mathcal O^\otimes$, $\mathcal O'^\otimes$, and $\mathcal O''^\otimes$ be quasi-operads.  A diagram
\[ \xymatrix{\mathcal O'^\otimes \times \mathcal O''^\otimes \ar[r] \ar[d] & \mathcal O^\otimes \ar[d] \\
\nerve(\Gamma) \times \nerve(\Gamma) \ar[r] & \nerve(\Gamma) } \]
\emph{exhibits} $\mathcal O^\otimes$ \emph{as a tensor product of} $\mathcal O'^\otimes$ and $\mathcal O''^\otimes$ if the induced map
\[ \mathcal O'^{\otimes, \natural} \odot \mathcal O''^{\otimes, \natural} \rightarrow \mathcal O^{\otimes, \natural} \]
is a weak equivalence of preoperads, i.e., if for every quasi-operad $\mathcal C^\otimes$,
\[ \Map^\sharp(\mathcal O, \mathcal C^{\otimes, \natural}) \rightarrow \Map^\sharp(\mathcal O' \odot \mathcal O'', \mathcal C^{\otimes, \natural}) \] is a homotopy equivalence.
\end{definition}

For a topological operad, there is a map $\mathcal O \otimes^L \mathcal O' \rightarrow \mathcal O \otimes^M \mathcal O'$.

\begin{definition}
Suppose that $\mathcal C$ and $\mathcal O$ are operads.  An $\mathcal O$-\emph{algebra} in $\mathcal C$ is a section $A \colon \mathcal O \rightarrow \mathcal C$, i.e,
\[ \xymatrix{ \mathcal O^\otimes \ar[rr] \ar[dr] && \mathcal C^\otimes \ar[dl] \\
& nerve(\Gamma) & } \] such that inert maps go to inert maps.
\end{definition}

Denote by $\Alg_\mathcal O(\mathcal C)$ the category of $\mathcal O$-algebras in $\mathcal C$.

\begin{definition}
The quasi-category $Alg_\mathcal O(\mathcal C)^\otimes$ represents the functor $PreOp \rightarrow Kan$ given by $Y \mapsto \Hom^\#_{PreOp}(Y \odot \mathcal O^{\otimes \natural}, \mathcal C^{\otimes \natural})$, where if $\mathcal O^\otimes$ is an operad, then $\mathcal O^{\otimes \natural}$ is a preoperad with $\mathcal E$ consisting of the inert maps.
\end{definition}

A map of quasi-operads is a map preserving inert arrows.  A fibration is a map which is a (categorical) fibration.

\begin{definition}
Let $p \colon \mathcal C \rightarrow \mathcal O$ be a fibration of operads and $u \colon \mathcal O' \rightarrow \mathcal O$ a morphism of operads.  An $\mathcal O$-\emph{algebra object} in $\mathcal C$ is the $\mathcal O'$-algebra object in $\mathcal O' \times_\mathcal O \mathcal C$, i.e., a map of operads $A \colon \mathcal O' \rightarrow \mathcal C$ such that $p \circ A=u$.
\end{definition}

\begin{example}
In conventional category theory, any category $\mathcal C$ defines an operad (with no $n$-ary operations for $n \neq 1$).  An algebra over such an operad, say in $\mathcal Top$, is just a functor $\mathcal C \rightarrow \mathcal Top$.
\end{example}

\begin{example}
The quasi-categorical version of the above example is as follows.  Let $Triv$ be the trivial operad.  A map $Triv \rightarrow \nerve(\Gamma)$ is an injection, so any operad $\mathcal O^\otimes \rightarrow \nerve(\Gamma)$ admits at most one morphism to $Triv$.  Existence of such a morphism means that all arrows in $\mathcal O^\otimes$ are inert.  A morphism $\mathcal O^\otimes \rightarrow \mathcal Top$ restricts to a map $\mathcal O^\otimes_{\langle 1 \rangle} \rightarrow \mathcal Top$.  Also, the map $\Fun^{lax}(\mathcal O^\otimes, \mathcal Top^\otimes) \rightarrow \Fun(\mathcal O^\otimes_{\langle 1\rangle}, \mathcal Top)$ is an acyclic fibration.
\end{example}

\begin{example}
Let $\mathcal O= \nerve(\Gamma)$ and $\mathcal C=\mathcal Top^\otimes$.  A map $\mathcal O \rightarrow \mathcal C$ assigns to $\langle n \rangle$ a collection of $n$ topological spaces (all homotopy equivalent to each other).  The arrow $\langle 2 \rangle \rightarrow \langle 1 \rangle$ carrying $1,2 \mapsto 1$ defines an operation which is commutative up to homotopy.  Thus, our notion of $\mathcal O$-algebra is automatically the ``up-to-homotopy" notion.
\end{example}

\begin{example}
Given a topological operad $\widetilde {\mathcal O}$, define the corresponding quasi-operad $\mathcal O^\otimes$.  It is an easy exercise to assign to any $\widetilde {\mathcal O}$-algebra in $\mathcal Top$ (in the conventional sense) a morphism of operads $\mathcal O^\otimes \rightarrow \mathcal Top^\otimes$ and therefore an $\mathcal O^\otimes$-algebra in the sense of our definition.
\end{example}

The quasi-category $Alg_\mathcal O(\mathcal C)$ of $\mathcal O$-algebras in an operad $\mathcal C$ has a canonical operad structure: the functor $\mathcal{PO}p Y \mapsto \Hom_{\mathcal{PO}p}(Y \odot \mathcal O^{\otimes, \natural}, \mathcal C^{\otimes, \natural})$ is representable by an operad which we denote $Alg_\mathcal O(\mathcal C)^\otimes$; it is even a symmetric monoidal quasi-category if $\mathcal C$ is.  The fiber of $Alg_\mathcal O(\mathcal C)^\otimes$ at $\langle 1 \rangle$ is precisely $Alg_\mathcal O(\mathcal C)$.

Now we can define $Alg_{\mathcal O_1}(Alg_{\mathcal O_2}(\mathcal C))$ and compare it to $Alg_{\mathcal O_1 \otimes \mathcal O_2}(\mathcal C)$:
\[ \begin{aligned}
\Hom_{\mathcal {PO}p}(Y, Alg_{\mathcal O_1}(Alg_{\mathcal O_2}(\mathcal C))) & = \Hom_{\mathcal{PO}p}(Y \odot \mathcal O_1^{\otimes, \natural}, Alg_{\mathcal O_2}(\mathcal C)^{\otimes, \natural}) \\
& = \Hom_{\mathcal {PO}p}(Y \otimes \mathcal O_1^{\otimes, \natural} \odot \mathcal O_2^{\otimes, \natural}, \mathcal C^{\otimes, \natural}) \\
& = \Hom_{\mathcal{PO}p}(Y, Alg_{\mathcal O_1 \otimes \mathcal O_2}(\mathcal C)).
\end{aligned} \]

Given two topological operads $\widetilde{\mathcal O_1}$ and $\widetilde{\mathcal O_2}$, we have a map
\[ \mathcal O_1 \otimes^L\mathcal O_2 \rightarrow \mathcal O_1 \otimes^M\mathcal O_2 \]
which need not be an equivalence in general.  For instance, $Assoc \otimes^M Assoc = Com$.  However, we believe that this is an equivalence if $\widetilde{\mathcal O_1}$ and $\widetilde{\mathcal O_2}$ are cofibrant.

The quasi-operads $\mathbb E[k]$ are defined as the ones constructed from the topological operads $\mathcal E[k]$.  The map $\mathbb E[k] \rightarrow N(\Gamma)$ is bijective on objects; the map
\[ \Map_{\mathbb E[k]}(\langle m \rangle, \langle n \rangle) \rightarrow \Map_{N(\Gamma)}(\langle m \rangle, \langle n \rangle) \] is $(k-1)$-connective.

\begin{example}
Let $\mathcal C^\otimes$ be a symmetric monoidal quasi-category.  The arrow $\varphi \colon \langle 0 \rangle \rightarrow \langle 1 \rangle$ gives $\varphi_! \colon \mathcal C^\otimes_{\langle 0 \rangle} \rightarrow \mathcal C^\otimes_{\langle 1 \rangle}$.  Since $\mathcal C^\otimes, {\langle 0 \rangle}$ is contractible, this defines an object $1$ in $\mathcal C^\otimes_{\langle 1 \rangle}$, uniquely up to equivalence.  An $\mathbb E[0]$-algebra $A$ in $\mathcal C^\otimes$ defines in particular a map $1 \rightarrow A$ in $\mathcal C^\otimes_{\langle 1 \rangle}$.  The claim is that
\[ \Fun^{lax}(\mathbb E[0], \mathcal C^\otimes) \rightarrow \mathcal C^\otimes_{\langle 1 \rangle} \] is an acyclic fibration.  In other words, an $\mathbb E[0]$-algebra is uniquely given, up to equivalence, by a pointed object in $\mathcal C^\otimes_{\langle 1 \rangle}$.
\end{example}

Assume that $\mathcal C^\otimes \rightarrow \nerve(\Gamma)$ is a symmetric monoidal quasi-category such that $\Map_{\mathcal C^\otimes}(x,y)$ are all $(n-1)$-truncated (have no higher homotopy groups).  Then we claim that $\Alg_{Com}(\mathcal C) \rightarrow \Alg_{\mathbb E[k]}(\mathcal C)$ is an equivalence for $k>n$.

\begin{example}
An $\mathbb E[2]$-algebra in $\Sets$ is already a commutative monoid.  An $\mathbb E[1]$-algebra in $\mathcal Cat$ is a monoidal category; an $\mathbb E[2]$-algebra in $\mathcal Cat$ is a braided monoidal category; an $\mathbb E[3]$-algebra in $\mathcal Cat$ is a symmetric monoidal category.
\end{example}

Define a map $\mathcal E[m] \times \mathcal E[n] \rightarrow \mathcal E[m+n]$ which assigns to a pair $\alpha \colon \square^m \times I \rightarrow \square^m$, $\beta \colon \square^n \times J \rightarrow \square^n$ .  The product map $\alpha \otimes \beta$ is $\square^{m+n} \times I \times J \rightarrow \square^{m+n}$.

\begin{theorem} (Dunn, '88)
The above described map induces and equivalence $\mathcal E[m] \otimes \mathcal E[n] \rightarrow \mathcal E[m+n]$.
\end{theorem}

We assume the similar result for quasi-operads $\mathbb E[n]$ can be deduced from Dunn's theorem, since $\mathbb E[k]$ are cofibrant.  Lurie gives an independent proof \cite{dag6}.

What is the correct definition of the center of an associative algebra?  Here, ``correct" means ``categorical" with the hope of generalization to higher categories.

\begin{definition}
Let $A$ be an associative algebra with unit.  Its \emph{center} is a universal object in the category of diagrams
\[ \xymatrix{ & A \otimes Z \ar[dr] \\
A \ar[ur]^{id \otimes 1} \ar[rr]^{id} && A.} \]
\end{definition}

We are going to give a quasi-categorical analogue for this construction, fit for $\mathbb E[k]$-algebras.  The center will automatically have the extra structure of an $\mathbb E[1]$-algebra commuting with the original structure.  Together with the additivity theorem and the identification of the center with the Hochschild complex, in the case $k=1$, gives the Deligne Conjecture.


\section{Models for $(\infty, n)$-categories}

An $(\infty, n)$-category is a generalization of an $(\infty, 1)$-category, where now $k$-morphisms are invertible for $k>n$.  The heuristic idea is that an $(\infty, n)$-category should be modeled by a category enriched in $(\infty, n-1)$-categories.

In fact, we have already used that one model for $(\infty, 1)$-categories is that of simplicial categories, or categories enriched in simplicial sets, and simplicial sets can be regarded as $(\infty, 0)$-categories, which are just $\infty$-groupoids.  If we want to generalize, we first need to have a closed monoidal category, so that composition makes sense.  If we want to do so in such a way that we get a model structure, we need at the very least that our model form a closed monoidal model category.  While simplicial categories do form a monoidal category under cartesian product, this structure is not in fact compatible with the model structure and therefore we cannot blindly proceed be induction.

Since quasi-categories form a monoidal model category with internal hom objects, we can enrich over them and obtain a model for $(\infty, 2)$-categories.  Similarly, we could take categories enriched in complete Segal spaces.  However, we will have the same problems as before if we then try to enrich over these enriched categories to continue the induction.

Other models for $(\infty, n)$-categories generalize some of our other approaches to $(\infty, 1)$-categories.  Because Segal categories and complete Segal spaces are simplicial objects, we can generalize them by considering multi-simplicial objects.  We first look at the $n$-fold complete Segal spaces of Barwick and Lurie.

We begin with $n$-fold simplicial spaces, or functors $(\Deltaop)^n \rightarrow \SSets$, which can equivalently be regarded as functors $\Deltaop \rightarrow \SSets^{(\Deltaop)^{n-1}}$.  An $n$-fold simplicial space is \emph{essentially constant} if there exists a levelwise weak equivalence $X' \rightarrow X$ where $X'$ is constant.

\begin{definition} \cite{lurietft}
An $n$-fold simplicial space $X \colon \Deltaop \rightarrow \SSets^{(\Deltaop)^{n-1}}$ is an $n$-\emph{fold Segal space} if:
\begin{enumerate}
\item each $X_m$ is an $n-1$-fold Segal space,

\item $X_m \rightarrow \underbrace{X_1 \times_{X_0} \cdots \times_{X_0} X_1}_m$ is a weak equivalence of $n-1$-fold Segal spaces, and

\item $X_0$ is essentially constant.
\end{enumerate}
\end{definition}

\begin{definition} \cite{lurietft}
An $n$-\emph{fold complete Segal space} is an $n$-fold Segal space such that:
\begin{enumerate}
\item each $X_m$ is a complete Segal space, and

\item the simplicial space $X_{m, 0, \ldots, 0}$ is a complete Segal space.
\end{enumerate}
\end{definition}

Similarly, one can define \emph{Segal} $n$-\emph{categories} as given by Hirschowitz-Simpson and Pelissier, which are $n$-fold simplicial spaces with such a Segal condition but with various discreteness conditions replacing the completeness conditions \cite{hs}, \cite{pel}.

While $1$-fold complete Segal spaces are just complete Segal spaces, the nice property that the model structure $\css$ is cartesian does not seem to be retained.  With the goal of finding a model which was still cartesian, Rezk developed his $\Theta_n$-spaces \cite{rezktheta}.

We first need to give an inductive definition of the category $\Theta_n$.  Let $\Theta_0$ be the terminal category with one object and no non-identity morphisms.  Then inductively define $\Theta_n$ to have objects $[m](c_1, \ldots, c_n)$ where $[m]$ is an object of ${\bf \Delta}$ and $c_1, \ldots, c_m$ are objects of $\Theta_{n-1}$.  One can think of these objects as strings of labeled arrows
\[ \xymatrix@1{0 \ar[r]^{c_1} & 1 \ar[r]^{c_2} & 2 \ar[r]^{c_3} & \cdots \ar[r]^{c_{m-1}} & m-1 \ar[r]^{c_m} & m.} \]  Morphisms
\[ [m](c_1, \ldots, c_m) \rightarrow [p](d_1 \ldots, d_p) \] are given by $(\delta, \{f_{ij}\})$ where $\delta \colon [m] \rightarrow [p]$ is a morphism of ${\bf \Delta}$ and $f_{ij} \colon c_i \rightarrow d_j$ is a morphism in $\Theta_{n-1}$, defined whenever $\delta(i-1)<j \leq \delta(i)$.  For example, if we take $\delta \colon [3] \rightarrow [4]$ defined by $0,1 \mapsto 1$, $2 \mapsto 3$, and $3 \mapsto 4$, then we would have $(\delta, f_{22}, f_{23}, f_{34}) \colon [3](c_1, c_2, c_3) \rightarrow [4](d_1, d_2, d_3, d_4)$.  Notice in particular that $\Theta_1={\bf \Delta}$.

The idea here is that objects of $\Theta_n$ are ``basic" strict $n$-categories in the same way that objects of ${\bf \Delta}$ are ``basic" categories.   If we consider functors $\Theta_n^{op} \rightarrow \SSets$, then there are ``Segal maps" giving composition in multiple levels, and we want to look at \emph{Segal space objects} for which these maps are weak equivalences.  While it is more difficult to explain, there is also a notion of ``completeness" for such functors, giving a definition of \emph{complete Segal objects}.

\begin{theorem} \cite{rezktheta}
There is a cartesian model category $\Theta_nSp$ on the category of functors $\Theta_n^{op} \rightarrow \SSets$ in which the fibrant objects are the complete Segal objects.
\end{theorem}

Now, we can enrich over this structure, using a result of Lurie \cite{lurie}.

\begin{theorem} (B-Rezk)
There is a model structure on the category of small categories enriched over $\Theta_{n-1}Sp$, with weak equivalences analogues of Dwyer-Kan equivalences.
\end{theorem}

\section{Finite approximations to $(\infty, 1)$-categories/homotopy theories}

\centerline{(Talk given by David Blanc, plus some additional notes)}

\vskip .1 in

In \cite{dksimploc}, Dwyer and Kan proposed weak equivalence classes of simplicially enriched categories as their candidate for a ``homotopy theory" in Quillen's sense \cite[I,\S 2]{quillen}: that is, a context for encoding all homotopy invariants of any given model category $\mathcal C$ (and no more).  This is a point of view, rather than a precise statement, mainly because we have no good definition of ``all homotopy invariants" of a model category.  Nevertheless, homotopy theories in this sense seem to be the best approximation available.

However, the model category $\mathcal{SC}at$ and all its Quillen equivalent versions share the difficulty that they are hard to work and compute with.  For practical purposes, we therefore need smaller versions, obtained either by limiting the objects or making the hom-complexes simpler.  The way to do the latter is to take $n$-Postnikov approximations, which is a monoidal functor in $\SSets$ and so extends to $\mathcal{SC}at$.  We can then try to produce alternative models for $(n,1)$-categories.

Let $\mathcal C$ be an $(\mathcal S, \mathcal O)$-\emph{category}, that is, a small simplicial category with object set $\mathcal O$.  We can apply any monoidal functor $F \colon \mathcal S \rightarrow \mathcal V$ to get a $(\mathcal V, \mathcal O)$-category.  For example, let $P^n$ be the $n$th Postnikov section $P^n \colon \mathcal S \rightarrow P^n\mathcal S$.  Then $P^n\mathcal C$ is a $(P^n\mathcal S, \mathcal O)$-category.

\begin{itemize}
\item $P^0 \mathcal C$ is equivalent to $\pi_0 \mathcal C$, an $\mathcal O$-category, but it is cofibrant.  (In simplicial categories, it looks like $F_*(\pi_0\mathcal C)$.)

\item For $n=1$, note that we have a pair of semi-adjoint functors
\[ \xymatrix@1{\widehat \pi_1 \colon \SSets \ar@<.5ex>[r] & \mathcal Gpds \colon \nerve \ar@<.5ex>[l]} \] which commute with products and thus extend to simplicial categories.
Then $P^1 \mathcal C$ is equivalent to $\widehat \pi_1 \mathcal C$ via the nerve in $(\mathcal Gpd, \mathcal O)-Cat$, or \emph{track categories} \cite{BW}.

\item $P^2\mathcal C$ is equivalent to a double track category (enriched in a certain type of double groupoid \cite{BPaolT}).

\item Higher-order models are only conjectural.
\end{itemize}

Note that $(\mathcal S, \mathcal O)$-categories have $k$-invariants, taking values in the $(\mathcal S, \mathcal O)$-cohomology of their Postnikov sections $P^nX$ with coefficients in the natural system (which is a $\widehat \pi_1X$-module) $\pi_{n+1}X$ \cite{dksimploc}.  Since all functors involved are monoidal (except cohomology), they pass from $\SSets$ to $(\mathcal S, \mathcal O)$.

These show the importance of $(P^n\mathcal S, \mathcal O)$-categories, since:
\begin{enumerate}

\item they allow us to replace one model by a (hopefully simpler) one with the same weak homotopy type, since it is constructed inductively using ``the same" $k$-invariants; and

\item they themselves carry homotopy-invariant information, as we will see later.
\end{enumerate}

Even $(P^n\mathcal S, \mathcal O)$-categories are too unwieldy to be useful, so we need to restrict the object sets, too.  This leads to the notion of a ``mapping algebra" \cite[\S 9]{BBlaC}:  let $\mathcal C$  be a simplicial model category and $A$ an object of $\mathcal C$.  Define $\mathcal C_A$ to be the full sub-simplicial category of $\mathcal C$ generated by $A$ under suspensions and coproducts of cardinality less than some $\lambda$.  Alternatively, we may take all (homotopy) colimits of bounded cardinality.

\begin{definition}
An $A$-\emph{mapping algebra} is a simplicial functor $\mathcal X \colon \mathcal C_A \rightarrow \mathcal S$.
\end{definition}

\begin{example}
Choose an object $Y$ of $\mathcal C$ and define $(\mathcal M_AY)(A')= \Map_\mathcal C(A',Y)$ for any object $A'$ of $\mathcal C_A$.  This is a \emph{realizable} mapping algebra.
\end{example}

Note that if $\mathcal O=\ob(\mathcal C_A)$ and $\mathcal O^+=\mathcal O \cup \{\ast\}$, then $\mathcal X$ is just an $(\mathcal S, \mathcal O^+)$-category such that $\mathcal X |_\mathcal O=\mathcal C_A$ and $\Map_\mathcal X(\ast, -)=\ast$.

\begin{definition}
An $A$-mapping algebra $\mathcal X$ is \emph{realistic} if $\mathcal X$ preserves the limits in $\mathcal C_A^{op}$.  In other words, $\mathcal X(\Sigma A')=\Omega \mathcal X(A')$ and $\mathcal X(\vee A_i)=\prod_i \mathcal X(A_i)$.
\end{definition}

So, a realistic $\mathcal X$ is determined by $\mathcal X(A)$.

As a motivating example, let $\mathcal C=Top$ and $A=S^1$.  Then $\mathcal M_{S^1}=\Omega Y$ with additional structure.  In particular, it has an $A_\infty$-structure.  The conclusion is that any realistic $S^1$-mapping algebra is realizable, i.e., $\mathcal X(A)=\Omega Y$ for some $Y$.  When $A=S^k$, $\mathcal M_AY$ is a $k$-fold loop space, equipped with an action of all mapping spaces between (wedges of) spheres on it and its iterated loops.  A realistic $A$-mapping algebra is any space $X$ equipped with such an action.  Since this includes in particular the $A_\infty$ (or $E(k)$) structure, we see that any realistic $A$-mapping algebra is realizable.

For more general $A$, even $A=S^1 \vee S^2$, the situation is more complicated; the purely ``algebraic" approach of analyzing the group structure does not work.  However, by enhancing the structure of of an $A$-mapping algebra suitably one can still recover $Y$ from $\mathcal X$, up to $A$-equivalence \cite[\S 10]{BBlaC}.

There is a dual version in which we map into $A$ and its products and loop spaces.  The realizable version is denoted by $\mathcal M^AY$.

\begin{example}
The space $\mathcal C^A$ may consist of all $K(V,n)$ with $V$ an $\mathbb F_p$-vector space.  Note that all mapping spaces here are generalized Eilenberg-Mac Lane spaces.
\end{example}

If we apply the $n$th Postnikov section functor to an $A$-mapping algebra, we obtain an $n-A$-\emph{mapping algebra}, which is defined to be a $(P^n\mathcal S,\mathcal O^+)$-category extending $P^n \mathcal C_A$ as above.  Again, it may or may not be $n$-\emph{realistic}.

\begin{example}
When $n=0$, this is essentially a $\pi_0\mathcal C_A$-algebra; if $A$ is a homotopy cogroup object, this algebra is an algebraic theory in the sense of Lawvere \cite{law}.
\end{example}

\begin{example}
Note that if $Y$ is a connected space, then $\pi_0 \Omega^kY=\pi_kY$, so for $A=S^1$, $P^0\mathcal M_AY$ encodes all homotopy groups of $Y$, as well as the action of all primary homotopy groups.
\end{example}

What can we say about an $n$-realistic $A$-mapping algebra?  (In other words, what is $P^n \mathcal X$ for $\mathcal X$ realistic, or something that behaves like it, as $\mathcal X$ may not exist?  It should be like truncating an operad at level $n$.)  The answer is that we recover an $n$-\emph{stem} for $Y$ \cite[\S 1]{BBlaS}.  For example, $P^{n+k}X \langle k-1 \rangle$ is the $(k-1)$-connective cover of $P^{n+k}X$ with only homotopy groups from $k$ to $n+k$.  We can't recover all of $Y$, since we can't deloop.

We know that the $E_2$-term of the Adams spectral sequence for a space $Y$ is $\Ext^{\ast\ast}_\mathcal A(H^*(Y; \mathbb F_2), \mathbb F_2)$, which depends only on $H^*(Y)$ as a module or algebra over the Steenrod algebra.  What about the higher terms?

\begin{theorem} \cite{BJibS,BBlaS}, \cite{BBlaH}
The $E_r$-term of the (stable or unstable) Adams spectral sequence for $Y$ is determined by the $(r-2)$nd order derived functors of $\mathcal M^AY$, where $A$ here is $\{K(\mathbb F_p,n)\}_{n=1}^\infty$, and these may depend in turn on $P^{n-2}\mathcal M^AY$, the $(r-2)$-dual stem.
\end{theorem}

\begin{remark}
At first glance it appears that $\mathcal M^AY$ is uninteresting in this case, since all the spaces in question are products of $\mathbb F_p$-Eilenberg-Mac Lane spaces, so their $k$-invariants are trivial.  But, this is not true of $\mathcal C^A$ as an $(\mathcal S,\mathcal O)$-category!
\end{remark}

\section{Application: The Cobordism Hypothesis}

The main application of $(\infty, n)$-categories at this time is Lurie's use of them in his proof of the cobordism hypothesis.  We begin with Baez-Dolan's original formulation of this conjecture.

Define a category $Cob(n)$ as follows.
\begin{itemize}
\item The objects are closed oriented $(n-1)$-dimensional manifolds $M$.

\item $\Hom(M,N)$ consists of diffeomorphism classes of cobordisms from $M$ to $N$, i.e., $n$-dimensional manifolds $B$ with $\partial B =\overline M \amalg N$, where $\overline M$ is the manifold $M$ with the opposite orientation.
\end{itemize}

Notice that $Cob(n)$ is symmetric monoidal under the disjoint union.  Let $Vect(k)$ be the category of vector spaces over a field $k$, and notice that this category is monoidal under the tensor product over $k$.  The following definition is due to Atiyah \cite{atiyah}.

\begin{definition}
A \emph{topological field theory of dimension} $n$ is a symmetric monoidal functor $Z \colon Cob(n) \rightarrow Vect(k)$.
\end{definition}

Notice that if $Z(M)=V$, then we must have $Z(\overline M)=V^\vee$, the dual vector space.  In particular, vector spaces in the image of $Z$ should be finite-dimensional.

In dimension 1, a topological field theory is completely determined by its value on a point, since all 0-dimensional manifolds are just collections of points.

In dimension 2, a topological field theory is a commutative Frobenius algebra.  The idea here is to break a cobordism up into manageable pieces along closed 1-dimensional manifolds.

We would like to do this kind of procedure in higher dimensions, but it doesn't work as well, since we may need to cut along manifolds with boundary or manifolds with corners.  This situation leads us to make a higher categorical version of $Cob(n)$.

\begin{definition}
The (weak) $n$-category $Cob(n)$ has:
\begin{itemize}
\item objects 0-dimensional manifolds,

\item 1-morphisms cobordisms between them,

\item 2-morphisms cobordisms between cobordisms,

\item $n$-morphisms diffeomorphism classes of cobordisms between $\cdots$ cobordisms between cobordisms.
\end{itemize}
\end{definition}

Then we define an $n$-\emph{extended topological field theory of dimension} $n$ to be a symmetric monoidal functor $Cob_n(n) \rightarrow \mathcal C$, where $\mathcal C$ is come symmetric monoidal (weak) $n$-category, which is some generalization of $Vect(k)$.  The images in $\mathcal C$ must be ``fully dualizable" in that they must have duals but also must have a notion of duals for morphisms up to level $n-1$.

\begin{theorem} (Baez-Dolan Cobordism Hypothesis)
Let $\mathcal C$ be a symmetric monoidal weak $n$-category.  Then the evaluation functor $Z \mapsto Z(\ast)$ defines a bijection between isomorphism classes of framed extended topological field theories and isomorphism classes of fully dualizable objects of $\mathcal C$.
\end{theorem}

This conjecture was proved as stated in dimension 1 by Baez and Dolan \cite{bd}, and in dimension 2 by Schommer-Pries \cite{sp}.  However, it gets difficult because weak $n$-categories get difficult to manage.

However, we still have lost information about our cobordisms by taking diffeomorphism classes at level $n$.  Heuristically, we can define instead an $(\infty, n)$-category $Bord_n$ with:
\begin{itemize}
\item objects 0-dimensional manifolds,

\item 1-morphisms cobordisms between them,

\item 2-morphisms cobordisms between cobordisms,

\item $n$-morphisms cobordisms between $\cdots$ cobordisms between cobordisms,

\item $n+1$-morphisms diffeomorphisms of the cobordisms of dimension $n$,

\item (n+2)-morphisms isotopies of diffeomorphisms, and so forth.
\end{itemize}

In fact, $Bord_n$ can be realized as an $n$-fold complete Segal space.  We actually want a more structured version, $Bord_n^{fr}$ where all manifolds are equipped with a framing.  The symmetric monoidal structure can be described roughly as described in Hinich's talk. The $(\infty, n)$-categorical version of the cobordism hypothesis was proved by Lurie, and can be ``truncated" to prove the original version.

\begin{theorem}
Let $\mathcal C$ be a symmetric monoidal $(\infty, n)$-category.  Then $Z \mapsto Z(\ast)$ determines a bijection between isomorphism classes of symmetric monoidal functors $Bord_n^{fr} \rightarrow C$ and isomorphism classes of fully dualizable objects of $\mathcal C$.
\end{theorem}

The main steps of the proof are as follows:
\begin{enumerate}
\item Give an inductive formulation: assume true for $(\infty, n-1)$-categories and describe the data needed to move up to $(\infty, n)$-categories and reduce the whole proof to this inductive step.

\item Reduce to the case where we have no additional structure on the manifolds (i.e., no framings).

\item Reformulate so that we can consider $(\infty, 1)$-categories rather than $(\infty, n)$-categories.

\item Use Morse theory to understand handle decompositions using a variation on $Bord_n$.

\item Show that the variation of $Bord_n$ is equivalent to $Bord_n$.
\end{enumerate}

\section{The cobordism $(\infty, n)$-category as an $n$-fold complete Segal space}

\centerline{(Additional talk given the following week at Hebrew University)}

\vskip .1 in

We want to know how cobordisms form an $(\infty, n)$-category, specifically in one of the precise models.  Here, we will look at a simpler case of an $(\infty, 1)$-category informally described as having:
\begin{itemize}
\item objects closed $(n-1)$-dimensional manifolds,

\item 1-morphisms cobordisms between them,

\item 2-morphisms diffeomorphisms,

\item 3-morphisms isotopies, etc.
\end{itemize}

How can this information be encoded in a complete Segal space?  Our treatment here is taken from \cite{lurietft}.

We begin with some notation.  Let $V$ be a $d$-dimensional vector space.  Let $Sub_0(V)$ be the space of closed submanifolds $M \subseteq V$ of dimension $(n-1)$, and let $Sub(V)$ be the space of compact $n$-dimensional manifolds $M \subseteq V \times [0,1]$ such that $\partial M = M \cap (V \times \{0,1\}$.

\begin{definition}
For $K \geq 0$, define the space
\[ SemiCob(n)_k^V = \{(t_0 < t_1 < \cdots < t_k; M) \} \]
where each $t_i \in \mathbb R$ and $M \subseteq [t_0, t_k]$ is either
\begin{itemize}
\item a smooth submanifold of dimension $(n-1)$, when $k=0$, or

\item a properly embedded submanifold of dimension $n$ which intersects each $V \times \{t_i\}$ transversally, for $k>0$.
\end{itemize}
\end{definition}

Notice that $SemiCob(n)^V_0$ is equivalent to $\mathbb R \times Sub_0(V)$, and that for $k>0$, $SemiCob(n)^V_k$ is an open subspace of $Sub(V) \times \{t_0 < \cdots <t_k\}$.

A strictly increasing map $f \colon [k] \rightarrow [k']$ in ${\bf \Delta}$ induces
\[ f^* \colon SemiCob(n)^V_{k'} \rightarrow SemiCob^V_k \] given by
\[ (t_0< \cdots, t_k; M) \mapsto (t_{f(0)}< \cdots < t_{f(k)}; M \cap V \times [t_{f(0)}, t_{f(k)}]). \]  Hence, we get a semi-simplicial space (with face maps but no degeneracy maps) $SemiCob(n)^V$. Then define $SemiCob(n)$ as the colimit of $SemiCob(n)^V$, where $V$ ranges over all finite-dimensional subspaces of some model of $\mathbb R^\infty$.

If we actually want a simplicial space (with degeneracy maps), it gets more complicated.

\begin{definition}
Let $PreCob(n)^V_k$ be the space $\{(t_0 \leq \cdots \leq t_k; M) \}$ with $t_i \in \mathbb R$ and $M \subseteq V \times \mathbb R$ a possibly non-compact $n$-dimensional manifold such that the projection $M \rightarrow \mathbb R$ has critical values disjoint from $\{t_0, \ldots, t_k\}$.  The subspace $PreCob^0(n)$ consists of the points actually satisfying $t_0 < \cdots < t_k$.
\end{definition}

Since we allow equality between the points, we have degeneracy maps which repeat them, and hence a simplicial space $PreCob(n)^V$.

There is a levelwise weak equivalence
\[ f \colon PreCob^0(n)^V \rightarrow SemiCob(n)^V \] given by
\[ (t_0 < \cdots <t_k; M) \mapsto (t_0 < \cdots < t_k; M \cap (V \times [t_0, t_k])). \]  Define $PreCob(n)$ to be the colimit of the $PreCob(n)^V$, as before.  then $PreCob(n)$ is a Segal space (possibly requiring Reedy fibrant replacement).

However, it is not necessarily complete: being so would violate the $s$-cobordism theorem.  To make it complete, we could apply a localization functor in $\css$.  However, it is not usually a problem here; since we are mapping out of it, we really only need for it to be cofibrant.

\end{document}